\documentclass[11pt,reqno,a4paper]{amsart}

\usepackage{docmute}

\usepackage{graphicx}
\usepackage{midpage}
\usepackage[margin=25truemm]{geometry}
\usepackage{amsmath,amssymb,amsthm,amsfonts}
\usepackage{tikz-cd}
\usepackage{setspace}
\usepackage{stix2}
\usepackage[protrusion=true,expansion=true,final]{microtype}
\usepackage{hyperref}
\usepackage{lscape}
\usepackage{enumitem}
\usepackage{caption}
\usepackage{array}
\usepackage{makecell}
\makeatletter
\@ifundefined{showhyphens}{}{}
\makeatother

\numberwithin{equation}{subsection}
\usetikzlibrary{calc}

\hypersetup{%
setpagesize=false,
 bookmarksnumbered=true,%
 bookmarksopen=true,%
 colorlinks=true,%
 linkcolor=red,
 citecolor=blue,
}

\let\it\relax

\newcommand{\it}[1]{\textit{#1}}

\newcommand{\op}{^\mathrm{op}}

\DeclareMathOperator{\gldim}{\mathrm{gl.dim}}
\DeclareMathOperator{\domdim}{\mathrm{dom.dim}}
\DeclareMathOperator{\codomdim}{\mathrm{codom.dim}}
\DeclareMathOperator{\pd}{\mathrm{pd}}
\DeclareMathOperator{\id}{\mathrm{id}}

\let\top\relax
\DeclareMathOperator{\top}{\mathrm{top}}
\DeclareMathOperator{\rad}{\mathrm{rad}}
\DeclareMathOperator{\soc}{\mathrm{soc}}
\DeclareMathOperator{\corad}{\mathrm{corad}}

\let\hom\relax

\newcommand{\hom}[2]{{}_{#1}(#2)}

\newcommand{\rhom}[2]{\mathbb{R}\mathrm{Hom}_{#1}(#2)}
\newcommand{\rend}[2]{\mathbb{R}\mathrm{End}_{#1}(#2)}
\newcommand{\lox}{\otimes^{\mathbb{L}}}

\newcommand{\MC}{\mathcal{C}}

\newcommand{\MI}{\mathcal{I}}
\newcommand{\MK}{\mathcal{K}}
\newcommand{\MS}{\mathcal{S}}
\newcommand{\MT}{\mathcal{T}}

\newcommand{\MP}{\mathcal{P}}

\newcommand{\MD}{\mathcal{D}}

\newcommand{\RK}{\mathrm{K}}
\newcommand{\RD}{\mathrm{D}}

\newcommand{\BE}{\mathbb{E}}

\newcommand{\tauleq}[1]{\tau^{\leq #1}}
\newcommand{\taugeq}[1]{\tau^{\geq #1}}

\newcommand{\dem}[1]{\mathcal{H}^d_{#1}}
\newcommand{\demdg}[1]{\mathscr{H}^d_{#1}}
\newcommand{\der}[1]{\mathcal{D}(#1)}

\newcommand{\htpy}[1]{\mathcal{K}(#1)}

\newcommand{\Ar}[3]{\ar[from=#1,to=#2,#3]}
\newcommand{\red}{\color{red}}
\newcommand{\tauAr}[3]{\ar[from={#1}, to={#2}, #3,color=gray, dash pattern=on 3pt off 4pt]}

\let\mod\relax

\DeclareMathOperator{\mod}{\mathrm{mod}}
\DeclareMathOperator{\pvd}{\mathrm{pvd}}
\DeclareMathOperator{\per}{\mathrm{per}}
\DeclareMathOperator{\Mod}{\mathrm{Mod}}

\DeclareMathOperator{\Cone}{\mathrm{Cone}}
\DeclareMathOperator{\Cocone}{\mathrm{Cocone}}

\DeclareMathOperator{\add}{\mathrm{add}}

\DeclareMathOperator{\End}{\mathrm{End}}

\theoremstyle{definition}

\newtheorem{theorem}{Theorem}[section]
\newtheorem*{theorem*}{Theorem}
\newtheorem{proposition}[theorem]{Proposition}
\newtheorem*{proposition*}{Proposition}

\newtheorem{lemma}[theorem]{Lemma}
\newtheorem{cor}[theorem]{Corollary}
\newtheorem*{cor*}{Corollary}

\newtheorem{definition}[theorem]{Definition}
\newtheorem*{definition*}{Definition}

\newtheorem{remark}[theorem]{Remark}
\newtheorem{example}[theorem]{Example}
\newtheorem{notation}[theorem]{Notation}

\newtheorem{fact}[theorem]{Fact}

\begin{document}

\title[Auslander correspondence for proper connective DG-algebras]{Auslander correspondence for proper connective DG-algebras \\via extended module categories}

\author{Nao Mochizuki}

\address{Graduate School of Mathematics, Nagoya University, Furocho, Chikusaku, Nagoya 464-8602, Japan}
\email{mochizuki.nao.n8@s.mail.nagoya-u.ac.jp}

\begin{abstract}
We establish a $d$-dimensional Auslander correspondence for $d$-truncated proper connective DG-algebras via $d$-extended module categories. A $d$-truncated proper connective DG-algebra $\Gamma$ is called \emph{Auslander} if its $d$-extended module category is a $d$-Auslander extriangulated category. Our main theorem gives a one-to-one correspondence: for any $d$-truncated proper connective DG-algebra $\Lambda$ whose $d$-extended module category has finitely many indecomposable objects, the endomorphism DG-algebra of an additive generator is Auslander, and conversely every Auslander $d$-truncated proper connective DG-algebra arises in this way.
\end{abstract}

\maketitle

\tableofcontents

\setcounter{section}{-1}

\section{Introduction}
The Auslander correspondence is a fundamental result in the representation theory of finite-dimensional algebras. It gives a homological characterization of finite representation type by relating an additive generator of the module category to the homological dimensions of its endomorphism algebra. The precise statement is as follows:
\begin{fact}[{\cite{aus1971}}]
There is a bijection between the following two classes:
\begin{enumerate}
    \item Morita-equivalence classes of finite-dimensional algebras $\Lambda$ of finite representation type.
    \item Morita-equivalence classes of finite-dimensional algebras $\Gamma$ with global dimension at most $2$ and dominant dimension at least $2$.
\end{enumerate}
The correspondence from (1) to (2) is given by taking the endomorphism algebra $\End_{\Lambda}M$ of an additive generator $M$ of $\mod \Lambda$ for a finite-dimensional algebra $\Lambda$ of finite representation type.
\end{fact}
A finite-dimensional algebra $\Gamma$ satisfying the homological conditions in (2) is called an \emph{Auslander algebra}. Subsequently, O.~Iyama introduced a higher-dimensional version of the Auslander correspondence \cite{MR2298820} by replacing additive generators with $d$-cluster tilting subcategories. The corresponding endomorphism algebras are the $d$-Auslander algebras, namely finite-dimensional algebras of global dimension at most $d+1$ and dominant dimension at least $d+1$.

Meanwhile, developments in silting theory \cite{MR2927802}, derived Morita theory \cite{kel1994}, and related areas have broadened the setting beyond finite-dimensional algebras to proper connective DG-algebras and their derived categories.

Since the derived category $\der{\Lambda}$ of a DG-algebra $\Lambda$ is in general more delicate, it is often convenient to work with the $d$-extended module category $\dem{\Lambda}$, which may be viewed as a ``module category'' over a truncated DG-algebra. This category is studied in \cite{gup2024,zho2024} {for finite-dimensional algebras} and \cite{moc2025b} {for proper connective DG-algebras}. We give its definition:
\begin{definition}[$=$Definition~\ref{def:extended-mod}]
Let $\Lambda$ be a proper connective DG-algebra. We define the \emph{$d$-extended module category} $\dem{\Lambda}$ of $\Lambda$ as the full subcategory of the perfect valued derived category $\pvd\Lambda$ consisting of objects $M\in\pvd\Lambda$ such that $H^i(M)=0$ {for every $i\notin[-d+1,0]$}.
\end{definition}
It coincides with the usual module category $\mod H^0(\Lambda)$ when $d=1$. If a proper connective DG-algebra $\Lambda$ satisfies $H^i(\Lambda)=0$ for any $i\leq-d$ (equivalently, $\Lambda\in\dem{\Lambda}$), we say that $\Lambda$ is \emph{$d$-truncated}.

For proper connective DG-algebras, the author showed in \cite{moc2025} that $d$-extended module categories admit abelian $d$-truncated DG-categories as DG-enhancements, which are DG-categories with properties analogous to $d$-dimensional generalizations of abelian categories (DG-analogues of abelian $(d,1)$-categories \cite{ste2023}). It is known that extended module categories are naturally extriangulated \cite{nak2019}, admit Auslander--Reiten sequences (see \cite{MR4133519} for $d$-self-injective DG-algebras, \cite{zho2024} for finite-dimensional algebras, and \cite{moc2025b} for general proper connective DG-algebras), and satisfy the first Brauer--Thrall theorem \cite{moc2025b}.

For finite-dimensional algebras, the $d$-extended module category is a tractable object, much like the module category itself. For example, results on one-to-one correspondence between {support $\tau$-tilting modules}, $2$-term silting complexes, and left-finite semibricks \cite{ada2014,MR4139031} have been generalized to extended module categories \cite{gup2024,zho2024,gup2025} by using $d$-extended module categories over finite-dimensional algebras and {$\tau_{[d]}$-tilting pairs in these categories}, $(d+1)$-term silting complexes, and left-finite semibricks in $d$-extended module categories.

These considerations motivate the study of the case where $d$-extended module categories have only finitely many indecomposable objects, as analogues of module categories of finite-dimensional algebras. {For the truncated Nakayama algebras considered in \cite{run2026}, the cases in which the $d$-extended module category $\dem{\Lambda}$ has finitely many indecomposable objects are classified there.}

In this paper, we introduce a $d$-dimensional analogue in a sense different from Iyama's higher Auslander correspondence (see Table~\ref{tab:comparison-iyama}). More precisely, we characterize, for a $d$-extended module category with finitely many indecomposable objects (in this case, we call $\Lambda$ \emph{representation-finite}), the homological properties of the endomorphism connective DG-algebra of its additive generator. This recovers the classical Auslander correspondence when $d=1$. We first define Auslander $d$-truncated connective DG-algebras.

\begin{definition}[$=$Definition~\ref{dfn:auslander}]
A $d$-truncated proper connective DG-algebra $\Gamma$ is called an \emph{Auslander $d$-truncated connective DG-algebra} if its $d$-extended module category $\dem{\Gamma}$ satisfies the following conditions:
\begin{enumerate}
\item The global dimension of $\dem{\Gamma}$ is at most $d+1$ as an extriangulated category.
\item The dominant dimension of $\dem{\Gamma}$ is at least $d+1$ as an extriangulated category.
\end{enumerate}
Or equivalently, $\dem{\Gamma}$ is a $d$-Auslander extriangulated category in the sense of \cite{che2023c}. {We include the zero category among $d$-Auslander categories by convention.}
\end{definition}
See \cite{gor2021,gor2023} for the notion of global dimension and dominant dimension of extriangulated categories. More precisely, see Definition~\ref{dfn:global dimension} and Definition~\ref{dfn:dominant_dim} for the definitions of these dimensions. The relationship between global dimension of connective DG-algebras and global dimension of their extended module categories is discussed in Proposition~\ref{prop:global_dim}. In particular, if $\dem{\Gamma}$ has global dimension at most $d+1$ as an extriangulated category, then $\Gamma$ itself has global dimension at most $2$ as a connective DG-algebra by this proposition.

To state the main result of this paper, we introduce the notion of Morita equivalence of $d$-truncated proper connective DG-algebras.

\begin{definition}[$=$Definition~\ref{dfn:morita}]
Let $\Lambda$ and $\Lambda'$ be proper $d$-truncated DG-algebras. We say that $\Lambda$ and $\Lambda'$ are \emph{Morita equivalent} if $\demdg{\Lambda}$ and $\demdg{\Lambda'}$ are quasi-equivalent as DG-categories. 
\end{definition}

The DG-categories $\demdg{\Lambda}$ and $\demdg{\Lambda'}$ are the canonical DG-enhancements of $\dem{\Lambda}$ and $\dem{\Lambda'}$ defined in Definition~\ref{dfn:can_enhance}.

The main result of this paper is the following:
\begin{theorem}[$=$Theorem~\ref{thm:main}]
There exists a one-to-one correspondence between the following two classes:
\begin{enumerate}
\item Morita-equivalence classes of $d$-truncated proper connective DG-algebras of finite representation type.
	\item Morita-equivalence classes of Auslander $d$-truncated proper connective DG-algebras.
\end{enumerate}
The correspondence from (1) to (2) is given by taking the endomorphism DG-algebra $\tauleq{0}\rend{\Lambda}{M}$ of an additive generator $M$ of $\dem{\Lambda}$ for a representation-finite $d$-truncated proper connective DG-algebra $\Lambda$. 
\end{theorem}

By this theorem, we can also characterize connective DG-categories which are quasi-equivalent to {canonical DG-enhancements} of $d$-extended module categories of representation-finite $d$-truncated proper connective DG-algebras as follows:

\begin{cor}[$=$Corollary~\ref{cor:main}]
Let $\mathscr{A}$ be a connective DG-category with idempotent complete homotopy category $H^0(\mathscr{A})$. Then the following are equivalent:
\begin{itemize}
\item[(1)] $\mathscr{A}$ is quasi-equivalent to {the canonical DG-enhancement $\demdg{\Lambda}$ of $\dem{\Lambda}$} for some $d$-truncated proper connective DG-algebra $\Lambda$ of finite representation type.
\item[(2)] The homotopy category $H^0(\mathscr{A})$ has an additive generator $M$ such that its endomorphism DG-algebra $\End_{\mathscr{A}}(M):=\mathscr{A}(M,M)$ is an Auslander $d$-truncated proper connective DG-algebra. 
\end{itemize}
\end{cor}

\begin{table}[h]
\centering
\renewcommand{\arraystretch}{1.4}
\setlength{\extrarowheight}{5pt}
\begin{tabular}{|
  >{\centering\arraybackslash}m{7.5cm}|
  >{\centering\arraybackslash}m{7.5cm}|}
\hline
\makecell[c]{Iyama's generalization\\ \cite{MR2298820}}
&
\makecell[c]{This paper}
\\ \hline
\makecell[c]{finite-dimensional algebra $\Lambda$}
&
\makecell[c]{$d$-truncated proper connective\\ DG-algebra $\Lambda$}
\\ \hline
\makecell[c]{module category $\mathrm{mod}\,\Lambda$}
&
\makecell[c]{$d$-extended module category $\mathcal{H}^d_\Lambda$\\(or its DG-enhancement)}
\\ \hline
\makecell[c]{$d$-cluster tilting object\\ $M \in \mathrm{mod}\,\Lambda$}
&
\makecell[c]{additive generator\\ $M\in\mathcal{H}^d_\Lambda$}
\\ \hline
\makecell[c]{endomorphism algebra\\ $\End_{\Lambda}(M)$}
&
\makecell[c]{endomorphism DG-algebra\\
$\tau^{\le 0}\mathbb{R}\End_{\Lambda}(M)$}
\\ \hline
\makecell[c]{$\mathrm{mod}\,\Gamma$ is $d$-Auslander\\
as an extriangulated category}
&
\makecell[c]{$\mathcal{H}^d_\Gamma$ is $d$-Auslander\\
as an extriangulated category}
\\ \hline
\makecell[c]{Classical Auslander correspondence\\ for the case $d=1$}
&
\makecell[c]{Classical Auslander correspondence\\ for the case $d=1$}
\\ \hline
\end{tabular}
\caption{Comparison between Iyama's higher Auslander correspondence and the correspondence in this paper.}
\label{tab:comparison-iyama}
\end{table}

\subsection*{Acknowledgements}
The author is grateful to Ryu Tomonaga for his suggestions of this research topic and for his valuable comments. The author would also like to thank Prof. Osamu Iyama and Prof. Hiroyuki Nakaoka for their helpful comments.

This work was supported by Make New Standard program at Nagoya University.

\section{\texorpdfstring{$d$-truncated DG-algebras and $d$-extended module categories}{d-truncated DG-algebras and d-extended module categories}}

This section briefly reviews $d$-truncated DG-algebras and $d$-extended module categories. For details, see \cite{moc2025,moc2025b}.

Throughout, fix a field $k$ and call a differential graded algebra over $k$ a DG-algebra.

\subsection{\texorpdfstring{The notion of $d$-extended module categories}{The notion of d-extended module categories}}

We begin by recalling basic definitions and properties of DG-algebras and their derived categories. These notions are standard; see \cite{kel1994,kel2006}.

\begin{definition}
Let $\Lambda$ be a DG-algebra.
\begin{enumerate}
	\item $\Lambda$ is \emph{connective} if $H^i(\Lambda)=0$ for any $i>0$.
	\item $\Lambda$ is \emph{$d$-truncated} if $H^i(\Lambda)=0$ for any $i\leq-d$.
	\item $\Lambda$ is \emph{proper} if $\dim_k H^*(\Lambda)<\infty$.
\end{enumerate}
\end{definition}

\begin{definition}
Let $\Lambda$ be a DG-algebra. 
\begin{itemize}
  \item The \emph{derived category} of $\Lambda$, denoted by $\der{\Lambda}$, is defined as the Verdier quotient of the homotopy category $\htpy{\Lambda}$ of right DG-modules over $\Lambda$ by the subcategory of acyclic right DG-modules. 
  \item The \emph{perfect derived category} of $\Lambda$, denoted by $\per\Lambda$, is defined as the thick subcategory of $\der{\Lambda}$ generated by $\Lambda\in\der{\Lambda}$. 
	\item The \emph{perfect valued derived category} of $\Lambda$, denoted by $\pvd\Lambda$, is defined as the full subcategory of the derived category $\der{\Lambda}$ consisting of objects $M\in\der{\Lambda}$ such that $\sum_{i\in\mathbb{Z}}\dim_k H^i(M)<\infty$. 
\end{itemize}
\end{definition}

\begin{remark}\label{rmk:t-str}
{Assume that $\Lambda$ is connective. Then $\der{\Lambda}$ has the standard $t$-structure $(\der{\Lambda}^{\leq0},\der{\Lambda}^{\geq0})$ given by:}
$$\der{\Lambda}^{\leq 0}=\{M\in\der{\Lambda}\mid {H^i(M)=0}\text{ for any }i>0\}$$
and the functor $H^0\colon \der{\Lambda} \to \Mod H^0(\Lambda)$ restricts to an equivalence between the heart of this $t$-structure and $\Mod H^0(\Lambda)$. {Moreover, we have} 
$$\der{\Lambda}^{\geq 0} = \{M\in\der{\Lambda}\mid H^i(M)=0\text{ for any }i<0\}.$$
If $\Lambda$ is proper connective, then the $t$-structure $\bigl(\der{\Lambda}^{\leq 0}, \der{\Lambda}^{\geq 0}\bigr)$ can be restricted to a $t$-structure $\bigl((\pvd\Lambda)^{\leq 0}, (\pvd\Lambda)^{\geq 0}\bigr)$ on $\pvd\Lambda$ and its heart is equivalent to $\mod H^0(\Lambda)$.
For more details, see \cite[Chapter~III]{bel2007}.
\end{remark}

\begin{lemma}\label{lem:homology_functor}
Let $\Lambda$ be a proper connective DG-algebra. Then the functor $H^0\colon \pvd\Lambda \to \mod H^0(\Lambda)$ induces a natural isomorphism for any $P\in\add \Lambda$ and $X\in\pvd\Lambda$:
$$\pvd\Lambda(P,X)\cong \mod H^0(\Lambda)\bigl(H^0(P),H^0(X)\bigr).$$
\end{lemma}
\begin{proof}
It holds for $P=\Lambda$, hence for any $P\in\add \Lambda$.
\end{proof}

\begin{notation}
Let $\Lambda$ be a proper connective DG-algebra. For any $M\in\pvd\Lambda$, write the truncation of $M$ as $\tau^{\leq i}M=\tau^{<i+1}M\in (\pvd\Lambda)^{\leq i}$ for any $i\in\mathbb{Z}$. Dually, write $\tau^{\geq i}M=\tau^{>i-1}M\in (\pvd\Lambda)^{\geq i}$ for any $i\in\mathbb{Z}$.
\end{notation}

\begin{definition}\label{k-dual_pvd}
Let $\Lambda$ be a proper connective DG-algebra. Define the \it{$k$-dual functor} by
$$D:=\rhom{k}{-,k}\colon (\pvd\Lambda)\op\to \pvd(\Lambda\op).$$
\end{definition}

\begin{remark}
The $k$-dual functor $D$ is an equivalence between $\pvd\Lambda$ and $\pvd\Lambda\op$. Moreover, we have a natural isomorphism $D(H^i(M)) \cong H^{-i}(D(M))$ for any $M \in \pvd\Lambda$.
\end{remark}

We now introduce $d$-extended module categories over proper connective DG-algebras, the main objects of study in this paper.

\begin{definition}\label{def:extended-mod}
Let $\Lambda$ be a proper connective DG-algebra and $d$ be a positive integer. The \emph{$d$-extended module category} of $\Lambda$, denoted by $\dem{\Lambda}$, is defined as the full subcategory of $\pvd\Lambda$ consisting of objects $M\in\pvd\Lambda$ such that $H^i(M)=0$ {for every $i\notin[-d+1,0]$}.
\end{definition}

\begin{remark}
By Remark~\ref{rmk:t-str}, the case $d=1$ recovers the usual module category $\mod H^0(\Lambda)$.
\end{remark}

\begin{remark}\label{rmk:reduced}
Let $\Lambda$ be a proper connective DG-algebra. {Replacing $\Lambda$ by the quasi-isomorphic DG-algebra $\tau^{\leq0}\Lambda$, we may assume that $\Lambda^i=0$ for $i>0$.} There is a DG-algebra homomorphism $\Lambda\to \tau^{>-d}\Lambda$ where $\tau^{>-d}\Lambda$ is the quotient DG-algebra $\Lambda/\tau^{\leq -d}\Lambda$ with DG-ideal $\tau^{\leq -d}\Lambda$. This induces a functor 
$$\tau^{\leq 0}\bigl(\rhom{\Lambda}{\tau^{>-d}\Lambda,-}\bigr)\colon \der{\Lambda}\to \der{\tau^{>-d}\Lambda}.$$
It is straightforward to check that this functor restricts to an equivalence $\dem{\Lambda}\to \dem{\tau^{>-d}\Lambda}$. Hence, without loss of generality, we may assume that any proper connective DG-algebra $\Lambda$ is $d$-truncated when considering its $d$-extended module category.
\end{remark}

Henceforth, fix a $d$-truncated proper connective DG-algebra $\Lambda$.

For the notion of Morita equivalence in Section~\ref{sec:auslander}, we introduce a canonical DG-enhancement of $\dem{\Lambda}$.

\begin{definition}\label{dfn:can_enhance}
Define the DG-category $\demdg{\Lambda}$ as the full DG-subcategory of $\tau^{\leq 0}\bigl(\mathscr{D}(\Lambda)\bigr)$ whose objects have {finite-dimensional total cohomology} concentrated in degrees $[-d+1,0]$. Here, the DG-category $\mathscr{D}(\Lambda)$ is the natural DG-enhancement of the derived category $\MD(\Lambda)$ (for example, the DG-category of $\mathrm{K}$-projective right DG-modules over $\Lambda$). Consequently, $H^0(\demdg{\Lambda})\cong \dem{\Lambda}$.
\end{definition}

\begin{remark}
The DG-category $\demdg{\Lambda}$ is an example of an abelian $d$-truncated DG-category introduced in \cite{moc2025}. This DG-category also admits the structure of an exact DG-category, as introduced in \cite{che2023, che2024,che2024b}. Thus, some properties of abelian $d$-truncated DG-categories can be applied to $\demdg{\Lambda}$ and its homotopy category $\dem{\Lambda}$.
\end{remark}

\begin{definition}
Define a functor $D_{d}\colon (\dem{\Lambda})\op \to \dem{\Lambda\op}$ as the composite of the $k$-dual functor $D\colon (\pvd\Lambda)\op\to \pvd(\Lambda\op)$ and the shift functor $[d-1]\colon \pvd(\Lambda\op)\to \pvd(\Lambda\op)$. Namely, for any $M\in\dem{\Lambda}$, set $D_d(M):=D(M)[d-1]$.
\end{definition}

\begin{proposition}
The category $\dem{\Lambda}$ carries a natural extriangulated structure:
\begin{enumerate}
\item Its extension $\BE(N,L)$ corresponds to $\der{\Lambda}(N,L[1])$ for any $N,L\in\dem{\Lambda}$,
\item Its conflations correspond to triangles in $\pvd\Lambda$ whose terms belong to $\dem{\Lambda}$,
\item It has enough projectives with projective objects $\add\Lambda$,
\item It has enough injectives with injective objects $\add D_d\Lambda$.
\end{enumerate}
\end{proposition}
\begin{proof}
See \cite[Remark~2.18]{nak2019} for (1) and (2), and \cite[Corollary~2.3]{moc2025b} for (3) and (4).
\end{proof}

\begin{notation}
We denote the projective objects of $\dem{\Lambda}$ by $\MP_\Lambda$. Dually, we denote the injective objects of $\dem{\Lambda}$ by $\MI_\Lambda$. Thus we have $\MP_\Lambda=\add \Lambda$ and $\MI_\Lambda=\add D_d\Lambda$ in $\pvd\Lambda$.
\end{notation}

For an extriangulated category $(\MC,\BE)$, higher extension functors $\BE^n(-,-)$ are defined in \cite{gor2021} for any $n\geq 1$. For $d$-extended module categories, they admit the following description:

\begin{proposition}\label{prop:ext}
For any $n\geq 1$ and $L,N\in\dem{\Lambda}$, there is an isomorphism 
$$\BE^n(N,L) \cong \pvd\Lambda(N,L[n]).$$
\end{proposition}
\begin{proof}
Because $\dem{\Lambda}$ has enough projectives, we can construct conflations of the form
\begin{equation}\label{eq:syzygy}
\Omega^{i+1} N\to P_i\to\Omega^i N\dashrightarrow
\end{equation}
where $0\leq i\leq n-1$ and $\Omega^0(N):=N$. By \cite[Corollary~{3.23}]{gor2021}, there exists an isomorphism $\BE^n(N,L)\cong \BE(\Omega^{n-1}N,L)$. Applying the cohomological functor $\pvd\Lambda(-,L[n-i-1])$ to \eqref{eq:syzygy} yields long exact sequences:
$$\pvd\Lambda(P_i,L[n-i-1]) \to \pvd\Lambda(\Omega^{i+1}N,L[n-i-1]) \to \pvd\Lambda(\Omega^i N,L[n-i]) \to \pvd\Lambda(P_i,L[n-i])$$
for any $0\leq i\leq n-1$. Because $P_i\in\MP_\Lambda$ and $L\in \dem{\Lambda}$, we have $\pvd\Lambda(P_i,L[>0])=0$. Hence there are isomorphisms $\pvd\Lambda(\Omega^{i+1}N,L[n-i-1])\cong \pvd\Lambda(\Omega^i N,L[n-i])$ for any {$0\leq i\leq n-2$}. Composing them gives
\begin{equation*}
\BE^n(N,L)\cong \BE(\Omega^{n-1}N,L)  \cong \pvd\Lambda(\Omega^{n-1}N,L[1]) \cong \pvd\Lambda(N,L[n]). \qedhere
\end{equation*}
\end{proof}

\begin{lemma}
Let $L,N\in\dem{\Lambda}$. Then $H^i\bigl(\rhom{\Lambda}{L,N}\bigr)=0$ for any $i\leq-d$. 
\end{lemma}
\begin{proof}
Because $\bigl((\pvd\Lambda)^{\leq 0},(\pvd\Lambda)^{\geq 0}\bigr)$ is a $t$-structure on $\pvd\Lambda$ by Remark~\ref{rmk:t-str}, we obtain the following isomorphisms:
$$
H^i\bigl(\rhom{\Lambda}{L,N}\bigr) \cong \pvd\Lambda(L,N[i]) \cong \pvd\Lambda\bigl(L,\tau^{\leq 0}(N[i])\bigr) 
$$
for any $i\in\mathbb{Z}$. If $i\leq -d$, then $\tau^{\leq 0}(N[i])=0$ because $N\in\dem{\Lambda}$. Therefore $H^i\bigl(\rhom{\Lambda}{L,N}\bigr)=0$ for any $i\leq -d$. 
\end{proof}

This lemma justifies defining the hom complex as follows:

\begin{definition}
Let $L,N\in\dem{\Lambda}$. We define the \emph{hom complex} $\hom{\Lambda}{L,N}\in\dem{k}$ from $L$ to $N$ as
$$\hom{\Lambda}{L,N}:=\tauleq{0}\bigl(\rhom{\Lambda}{L,N}\bigr).$$
This construction yields a bilinear functor $\hom{\Lambda}{-,-}\colon (\dem{\Lambda})\op\times \dem{\Lambda}\to \dem{k}$. For a morphism $f\colon M\to N$ and $L\in\dem{\Lambda}$, we denote the induced morphism by 
$$f_*\colon \hom{\Lambda}{L,M}\to \hom{\Lambda}{L,N} \quad\text{ and }\quad f^*\colon \hom{\Lambda}{N,L}\to \hom{\Lambda}{M,L}.$$
\end{definition}

\begin{remark}\label{rmk:0-th_hom}
Observe that the $0$-th cohomology of the hom complex $\hom{\Lambda}{L,N}$ corresponds to the morphism space $\dem{\Lambda}(L,N)$ for any $L,N\in\dem{\Lambda}$. 
\end{remark}

\begin{definition}\label{dfn:end_alg}
Let $M\in\dem{\Lambda}$. Define the \it{endomorphism $d$-truncated DG-algebra} $\End_\Lambda(M)$ of $M$ as $\tau^{\leq 0}\bigl(\rend{\Lambda}{M}\bigr)$ where $\rend{\Lambda}{M}$ is the derived endomorphism DG-algebra of $M$.
\end{definition}

\begin{definition}
Let $f\colon L\to N$ be a morphism in $\dem{\Lambda}$ and $1\leq n\leq d$. We say that $f$ is an \emph{$n$-monomorphism} if the induced morphism $H^{i}f\colon H^{i}L\to H^{i}N$ is {an} isomorphism for any $i \leq -n$ and $H^{-n+1}f$ is a monomorphism in $\mod H^0\Lambda$. Dually, we say that $f$ is an \emph{$n$-epimorphism} if the induced morphisms $H^{i}f\colon H^{i}L\to H^{i}N$ are isomorphisms for any $i\geq -d+n+1$ and $H^{-d+n}f$ is an epimorphism in $\mod H^0\Lambda$.
\end{definition}

\begin{lemma}\label{lem:n-mono_epi_chara}
Let $f\colon M\to N$ be a morphism in $\dem{\Lambda}$ and $1\leq n\leq d$. Then the following are equivalent:
\begin{enumerate}
\item The morphism $f$ is an $n$-monomorphism.
\item The cocone $\Cocone f$ of $f$ satisfies $H^{\leq -n+1}(\Cocone f)=0$.
\item The induced morphism $f_*\colon \hom{\Lambda}{L,M}\to \hom{\Lambda}{L,N}$ is an $n$-monomorphism for any $L\in\dem{\Lambda}$.
\end{enumerate}
Dually, the following are equivalent:
\begin{enumerate}
\item The morphism $f$ is an $n$-epimorphism.
\item The cone $\Cone f$ of $f$ satisfies $H^{\geq -d+n}(\Cone f)=0$.
\item The induced morphism $f^*\colon \hom{\Lambda}{N,L}\to \hom{\Lambda}{M,L}$ is an $n$-monomorphism for any $L\in\dem{\Lambda}$.
\end{enumerate}
\end{lemma}
\begin{proof}
The equivalence between (1) and (2) is clear. The equivalence between (1) and (3) follows from \cite[Proposition~3.47]{moc2025}.
\end{proof}

The following corollaries are immediate consequences of Lemma~\ref{lem:n-mono_epi_chara} and definitions.

\begin{cor}\label{cor:defl_chara}
Let $f\colon L\to N$ be a morphism in $\dem{\Lambda}$. Then the following hold:
\begin{enumerate}
\item {Every $1$-monomorphism in $\dem{\Lambda}$ is a monomorphism.}
\item The following are equivalent:
\begin{enumerate}
	\item $f$ is a $d$-monomorphism,
	\item $H^{-d+1}f$ is a monomorphism in $\mod H^0\Lambda$,
	\item the cone $\Cone f$ of $f$ belongs to $\dem{\Lambda}$,
	\item $f$ is an inflation.
\end{enumerate}
\end{enumerate}
\end{cor}
\begin{proof}
{(1) If $f$ is a $1$-monomorphism, then $\Cocone f\in(\pvd\Lambda)^{\geq1}$ by Lemma~\ref{lem:n-mono_epi_chara}. Thus
\[
\pvd\Lambda(X,\Cocone f)=0
\qquad\bigl(X\in\dem{\Lambda}\subset(\pvd\Lambda)^{\leq0}\bigr).
\]
Applying $\pvd\Lambda(X,-)$ to the triangle
\[
\Cocone f\to L\xrightarrow{f}N\to(\Cocone f)[1]
\]
shows that $f$ is a monomorphism. The equivalences in (2) follow directly from Lemma~\ref{lem:n-mono_epi_chara}.}
\end{proof}

\begin{cor}\label{cor:balanced}
Let $f\colon M\to N$ be a morphism in $\dem{\Lambda}$. If $f$ is both a $d$-epimorphism and a $1$-monomorphism, then it is an isomorphism.

Dually, if $f$ is a $1$-epimorphism and a $d$-monomorphism, then it is an isomorphism.
\end{cor}
\begin{proof}
This is immediate from Lemma~\ref{lem:n-mono_epi_chara}.
\end{proof}

Since $d$-epimorphisms in $\dem{\Lambda}$ can be characterized by deflations, we have the following lemmas:
\begin{lemma}\label{lem_chara_proj_inj}
 For any object $X$ in $\dem{\Lambda}$, the following hold:
\begin{enumerate}
\item $X$ is projective in $\dem{\Lambda}$ if and only if the functor $\hom{\Lambda}{X,-}\colon \dem{\Lambda}\to \dem{k}$ sends $d$-epimorphisms to $d$-epimorphisms.
\item $X$ is injective in $\dem{\Lambda}$ if and only if the functor $\hom{\Lambda}{-,X}\colon {(\dem{\Lambda})\op}\to \dem{k}$ sends $d$-monomorphisms to $d$-epimorphisms.
\end{enumerate}
\end{lemma}
\begin{proof}
We show the first assertion. By Remark~\ref{rmk:0-th_hom} and Corollary~\ref{cor:defl_chara}, the morphism $f_*\colon \hom{\Lambda}{X,L}\to \hom{\Lambda}{X,N}$ is a $d$-epimorphism if and only if the morphism 
$$f\circ -\colon \dem{\Lambda}(X,L)\to \dem{\Lambda}(X,N)$$
is surjective for any $f\colon L\to N$ in $\dem{\Lambda}$. Thus, the assertion follows from the definition of projective objects in extriangulated categories and Corollary~\ref{cor:defl_chara}. 
\end{proof}

Next, we introduce the notions of tops, radicals, socles, and coradicals in $\dem{\Lambda}$ as in usual module categories.

\begin{definition}\label{dfn:top_radical}
The top of $M\in\dem{\Lambda}$, denoted $\top M$, is defined as the top of $H^0 M$ in $\mod H^0\Lambda\subset \pvd\Lambda$. The radical of $M$ is defined as the cocone of the composite morphism $M \to H^0(M) \to \top M$. Note that there exists a conflation in $\dem{\Lambda}$ as follows:
$$\rad M \to M \to \top M \dashrightarrow$$

Dually, the socle of $M\in\dem{\Lambda}$, denoted $\soc M$, is defined as $(\soc_{H^0\Lambda}H^{-d+1}(M))[d-1]$ in $\pvd\Lambda$. The coradical of $M$ is defined as the cone of the canonical morphism $\soc M \to M$. Note that there exists a conflation in $\dem{\Lambda}$ as follows:
$$\soc M \to M \to \corad M \dashrightarrow$$
\end{definition}

\begin{remark}\label{rmk:rad_1_mono}
By the definition of radicals, the morphism $\rad M\to M$ is a $1$-monomorphism for any $M\in\dem{\Lambda}$. Dually, the morphism $M\to \corad M$ is a $1$-epimorphism for any $M\in\dem{\Lambda}$.
\end{remark}

{If $f\colon L\to N$ is a $d$-epimorphism, then $\Cocone f\in\dem{\Lambda}$ and there is a conflation $\Cocone f\to L\to N\dashrightarrow$. Dually, if $f$ is a $d$-monomorphism, then $\Cone f\in\dem{\Lambda}$ and there is a conflation $L\to N\to\Cone f\dashrightarrow$.}

\begin{lemma}
The category $\dem{\Lambda}$ is a hom-finite Krull-Schmidt category.
\end{lemma}
\begin{proof}
Since $\pvd\Lambda$ is a hom-finite triangulated category with split idempotents \cite[Proposition~6.12]{MR3898985}, $\pvd\Lambda$ is a Krull-Schmidt category. Thus, its full subcategory $\dem{\Lambda}$ is also a hom-finite Krull-Schmidt category because it is closed under direct summands in $\pvd\Lambda$.
\end{proof}

In the rest of this subsection, we briefly recall the Auslander-Reiten theory of $d$-extended module categories established in \cite{moc2025b}.

\begin{definition}
Let $f\colon M\to N$ be a morphism in $\dem{\Lambda}$. 
\begin{itemize}
	\item $f$ is \emph{left minimal} if any morphism $g\colon N\to N$ satisfying $g\circ f=f$ is an isomorphism. Dually, $f$ is \emph{right minimal} if any morphism $g\colon M\to M$ satisfying $f\circ g=f$ is an isomorphism. 
	\item $f$ is \emph{left almost split} if it is not a section and for any morphism $h\colon M\to L$ which is not a section, there exists a morphism $h'\colon N\to L$ such that $h=h'\circ f$. Dually, $f$ is right almost split if it is not a retraction and for any morphism $h\colon L\to N$ which is not a retraction, there exists a morphism $h'\colon L\to M$ such that $h=f\circ h'$. 
	\item $f$ is a \emph{sink morphism} (resp. \emph{source morphism}) if it is a right minimal and right almost split (resp. left minimal and left almost split) morphism.
\end{itemize}
\end{definition}

\begin{definition}[{\cite[Theorem~2.9]{iya2024}}]
A conflation $L\xrightarrow{f} M\xrightarrow{g} N\dashrightarrow$ in $\dem{\Lambda}$ is called an \emph{Auslander-Reiten conflation} if it satisfies any of the following equivalent conditions:
\begin{itemize}
	\item The morphism $f$ is a source morphism.
	\item The morphism $f$ is a left almost split morphism and $N$ is indecomposable.
	\item The morphism $g$ is a sink morphism.
	\item The morphism $g$ is a right almost split morphism and $L$ is indecomposable.
\end{itemize}
\end{definition}

\begin{theorem}[{\cite[Corollary~1.3]{moc2025b}}]\label{thm:existance_AR_seq}
The category $\dem{\Lambda}$ has Auslander-Reiten conflations, that is, for any indecomposable non-projective object $N\in\dem{\Lambda}$, there exists an Auslander-Reiten conflation with end term $N$. Dually, for any indecomposable non-injective object $L\in\dem{\Lambda}$, there exists an Auslander-Reiten conflation with starting term $L$.
\end{theorem}

\begin{proposition}[{\cite[{Propositions~3.7 and~3.8}]{moc2025b}}]
For any indecomposable projective object $P\in\dem{\Lambda}$, the morphism $\rad P\to P$ is a sink morphism. Dually, for any indecomposable injective object $I\in\dem{\Lambda}$, the morphism $I\to \corad I$ is a source morphism.
\end{proposition}

\subsection{The notion of presentations and copresentations}

\begin{definition}\label{dfn:presentations}
For any {non-negative} integer $n$, we define the \emph{category of $(n+1)$-term projective presentations} $\MK^{[-n,0]}(\MP_\Lambda)$ to be the full subcategory of $\per\Lambda\subset \pvd\Lambda$ as follows:
$$\MK^{[-n,0]}(\MP_{\Lambda}):=\MP_{\Lambda} *\MP_{\Lambda}[1]*\cdots *\MP_{\Lambda}[n]\subset\per\Lambda\subset\pvd\Lambda,$$
and also denote by $\MK^{[-n,0]}(\MP_{\Lambda}\cap\MI_{\Lambda})$ the full subcategory of $\MK^{[-n,0]}(\MP_\Lambda)$ given by the full subcategory:
$$\MK^{[-n,0]}(\MP_{\Lambda}\cap\MI_{\Lambda}):=(\MP_{\Lambda}\cap\MI_{\Lambda})*(\MP_{\Lambda}\cap\MI_{\Lambda})[1]*\cdots *(\MP_{\Lambda}\cap \MI_{\Lambda})[n]{.}$$
{We} call it the \emph{category of $(n+1)$-term projective-injective presentations}. If $n=d$, we simply call the category $\MK^{[-d,0]}(\MP_\Lambda)$ the \emph{category of projective presentations}.
\end{definition}

We also introduce the dual notion.

\begin{definition}
For any $n\in\mathbb{Z}_{\geq 0}$, we define the \emph{category of $(n+1)$-term injective copresentations} $\MK^{[0,n]}(\MI_\Lambda)$ to be the full subcategory of {$\pvd\Lambda$} as follows:
$$\MK^{[0,n]}(\MI_{\Lambda}):=\MI_{\Lambda}[-n]*\MI_{\Lambda}[-n+1]*\cdots *\MI_{\Lambda}{\subset\pvd\Lambda},$$
and also denote by ${\MK^{[0,n]}(\MP_\Lambda\cap\MI_\Lambda)}$ the full subcategory of $\MK^{[0,n]}(\MI_\Lambda)$ given by the full subcategory:
$$\MK^{[0,n]}(\MP_\Lambda\cap\MI_\Lambda):=(\MP_\Lambda\cap\MI_{\Lambda})[-n]*(\MP_\Lambda\cap\MI_{\Lambda})[-n+1]*\cdots *(\MP_\Lambda\cap \MI_{\Lambda}){.}$$
{We} call it the \emph{category of $(n+1)$-term projective-injective copresentations}. If $n=d$, we simply call the category $\MK^{[0,d]}(\MI_\Lambda)$ the \emph{category of injective copresentations}.
\end{definition}

\begin{remark}
We note the following properties of objects in $\MK^{[-d,0]}(\MP_\Lambda)$ and $\MK^{[0,d]}(\MI_\Lambda)$:
\begin{itemize}
	\item Since $\Lambda$ satisfies $H^{>0}(\Lambda[i])=0$ and $H^{<-2d+1}(\Lambda[i])=0$ for any $0\leq i\leq d$, we have $H^{>0} (P^\bullet)=0$ and $H^{<-2d+1} (P^\bullet)=0$ for all $P^\bullet\in \MK^{[-d,0]}(\MP_\Lambda)$. Thus, $\tau^{> -d} P^\bullet$ and $(\tau^{\leq -d} P^\bullet)[-d]$ are in $\dem{\Lambda}$.
	\item Dually, we also have $H^{>d} I^\bullet=0$ and $H^{<-d+1} I^\bullet=0$ for all $I^\bullet\in \MK^{[0,d]}(\MI_\Lambda)$. Thus, $\tau^{\leq 0} I^\bullet$ and $(\tau^{> 0} I^\bullet)[d]$ are in $\dem{\Lambda}$.
\end{itemize}
\end{remark}

\begin{remark}
{The shifted derived Nakayama functor $(-\lox_\Lambda D\Lambda)[-1]\colon\RD(\Lambda)\to\RD(\Lambda)$ restricts to an equivalence $\RK^{[-d,0]}(\MP_\Lambda)\to\RK^{[0,d]}(\MI_\Lambda)$, since the following holds in $\pvd\Lambda$:}
$${(\Lambda\lox_\Lambda D\Lambda)[-1]\cong D\Lambda[-1]\cong(D_d\Lambda)[-d].}$$
\end{remark}

\begin{definition}\label{dfn:proj_pres_obj}
Let $N$ be an object in $\dem{\Lambda}$. A \emph{projective presentation} of $N$ is an object $P^\bullet\in \MK^{[-d,0]}(\MP_\Lambda)$ such that $\tau^{> -d} P^\bullet \cong N$ in $\dem{\Lambda}$. Dually, an \emph{injective copresentation} of $N$ is an object $I^\bullet\in \MK^{[0,d]}(\MI_\Lambda)$ such that $\tau^{\leq 0} I^\bullet \cong N$ in $\dem{\Lambda}$.
\end{definition}

\begin{proposition}[{\cite[Theorem~2.5 {and Proposition~2.10}]{moc2025b}}]\label{prop:proj_inj_present}
Consider the functors
\begin{align*}
\tau^{>-d}&\colon\MK^{[-d,0]}(\MP_\Lambda)\to\dem{\Lambda},\\
\tau^{\leq0}&\colon\MK^{[0,d]}(\MI_\Lambda)\to\dem{\Lambda}.
\end{align*}
These functors induce equivalences
$$\MK^{[-d,0]}(\MP_\Lambda)/\MP_\Lambda[d]\xrightarrow{\sim} \dem{\Lambda} \quad\text{and}\quad\MK^{[0,d]}(\MI_\Lambda)/\MI_\Lambda[-d]\xrightarrow{\sim} \dem{\Lambda}$$
where $\MK^{[-d,0]}(\MP_\Lambda)/\MP_\Lambda[d]$ is the ideal quotient of $\MK^{[-d,0]}(\MP_\Lambda)$ by the ideal consisting of morphisms factoring through objects in $\MP_\Lambda[d]$ and $\MK^{[0,d]}(\MI_\Lambda)/\MI_\Lambda[-d]$ is defined similarly.
\end{proposition}

\begin{lemma}\label{lem:proj_pres}
The following hold:
\begin{enumerate}
\item Let $P^\bullet\in \MK^{[-d,0]}(\MP_\Lambda)$. Then $\pvd\Lambda(P^\bullet, M[i])=0$ for any $M\in\dem{\Lambda}$ and $i\geq d+1$. 
\item Let $I^\bullet\in \MK^{[0,d]}(\MI_\Lambda)$. Then $\pvd\Lambda(M, I^\bullet[i])=0$ for any $M\in\dem{\Lambda}$ and $i\geq d+1$.
\end{enumerate}
\end{lemma}
\begin{proof}
We only show the first statement. The second assertion follows by a dual argument. Let $i\geq d+1$ be an integer. Since $P^\bullet$ belongs to $\add\Lambda *\add\Lambda[1]*\cdots *\add\Lambda[d]$, it suffices to show that $\pvd\Lambda(\Lambda[j], M[i])=0$ for any $0\leq j\leq d$. This is clear since $H^{> 0} M=0$.
\end{proof}

The following proposition gives a connection between presentations and homological dimensions of $\dem{\Lambda}$ as an extriangulated category.

\begin{proposition}\label{prop:inductive_cocone}
For any $N\in\dem{\Lambda}$ and full subcategory $\MC\subset\dem{\Lambda}$, the following are equivalent:
\begin{enumerate}
\item There exist conflations in $\dem{\Lambda}$ as follows:
$$L^{0}\xrightarrow{m^{0}} M^{0} \xrightarrow{e^{0}} N\dashrightarrow;$$
$$L^{-1}\xrightarrow{m^{-1}} M^{-1} \xrightarrow{e^{-1}} L^0\dashrightarrow;$$
$$\vdots$$
$$L^{-d}\xrightarrow{m^{-d}} M^{-d} \xrightarrow{e^{-d}} L^{-d+1}\dashrightarrow$$
with ${M^{-i}}\in\MC$ for any $0\leq i\leq d$.
\item There exists an object $M^\bullet\in \MC*\MC[1]*\cdots * \MC[d]\subset \pvd\Lambda$ such that $\tau^{> -d} M^\bullet \cong N$ in $\dem{\Lambda}$.
\end{enumerate}
\end{proposition}
\begin{proof}
The proof is similar to \cite[Lemma~{2.6}]{moc2025b}.
\end{proof}

\section{\texorpdfstring{Auslander correspondence for $d$-truncated DG-algebras}{Auslander correspondence for d-truncated DG-algebras}}\label{sec:auslander}

\subsection{\texorpdfstring{Homological dimensions of $d$-truncated DG-algebras}{Homological dimensions of d-truncated DG-algebras}}

In this subsection, we discuss homological dimensions of $d$-truncated DG-algebras. Throughout this subsection, we fix a $d$-truncated proper connective DG-algebra $\Lambda$. 

\begin{definition}
Let $\MT$ be a triangulated category and $\MS$ be a full subcategory of $\MT$. We denote by $\langle \MS \rangle$ the smallest full subcategory of $\MT$ which contains $\MS$ and is closed under isomorphisms, finite direct sums, and extensions. 
\end{definition}

\begin{lemma}\label{lem:simple_modd}
The following equation holds as a full subcategory of $\pvd\Lambda$:
$${\dem{\Lambda}=\langle\add\top\Lambda\rangle[d-1]*\cdots*\langle\add\top\Lambda\rangle[1]*\langle\add\top\Lambda\rangle}$$
\end{lemma}
\begin{proof}
Since $\mod H^0({\Lambda})$ corresponds to $\langle\add \top\Lambda\rangle$ in $\pvd\Lambda$, we have the following equation:
\begin{align*}
\dem{\Lambda}=&{(\mod H^0\Lambda)[d-1]*\cdots*(\mod H^0\Lambda)[1]*\mod H^0\Lambda}\\
=&{\langle\add\top\Lambda\rangle[d-1]*\cdots*\langle\add\top\Lambda\rangle[1]*\langle\add\top\Lambda\rangle}\qedhere
\end{align*}
\end{proof}

We define the global dimension of DG-algebras. The definition of global dimension is well-known. However, we here adopt the definition given in \cite{tom2025} and \cite{min2021} since it is more suitable for our purpose.

\begin{definition}[{\cite[Proposition-Definition~3.3.]{tom2025}, {\cite[Theorem~5.1]{min2021}}}]\label{dfn:global dimension}
Let $n$ be a non-negative integer. We say that \emph{$\Lambda$ has global dimension at most $n$}, denoted $\gldim \Lambda\leq n$, if the following holds.
$$\BE^{> n}(\top\Lambda,\top\Lambda)\cong\pvd\Lambda(\top\Lambda,\top\Lambda[>n])=0$$
(the first equality follows from Proposition~\ref{prop:ext}). 
We define the \emph{global dimension} $\gldim \Lambda$ of $\Lambda$ as the infimum of the set of {non-negative} integers $n$ such that $\gldim \Lambda\leq n$. If there is no such integer, we set $\gldim \Lambda=\infty$.
\end{definition}

\begin{remark}
For an extriangulated category $(\MC,\BE)$, the global dimension $\gldim \MC$ of $(\MC,\BE)$ is defined as the supremum of {$\{0\}\cup\{n\geq1\mid\BE^n(M,N)\neq0\text{ for some }M,N\in\MC\}$} (see \cite[Definition~{3.30}]{gor2021} for details). In general, global dimension of $\dem{\Lambda}$ as extriangulated category does not coincide with the global dimension of $d$-truncated DG-algebras defined in Definition~\ref{dfn:global dimension} except in the case $d=1$. However, these two notions of global dimension are closely related as in Proposition~\ref{prop:global_dim}.
\end{remark}

\begin{proposition}\label{prop:global_dim}
{If $H^0(\Lambda)\neq0$, then} $\gldim\Lambda=\gldim(\dem{\Lambda})-d+1$. Moreover, if either of these is infinite, then the other is also infinite. {If $H^0(\Lambda)=0$, then $\gldim\Lambda=\gldim(\dem{\Lambda})=0$.}
\end{proposition}
\begin{proof}
{If $H^0(\Lambda)=0$, connectivity implies that $\Lambda$ is acyclic, so $\pvd\Lambda=0$ and both global dimensions are zero. Otherwise, the assertion follows from Lemma~\ref{lem:simple_modd} and the definitions of global dimension.}
\end{proof}

\begin{definition}\label{dfn:pd}
Let $M\in\dem{\Lambda}$. {We define the \emph{projective dimension} $\pd M$ and the \emph{injective dimension} $\id M$ by $\pd M:=\sup\bigl(\{0\}\cup\{n\geq1\mid\BE^n(M,-)\neq0\}\bigr)$ and $\id M:=\sup\bigl(\{0\}\cup\{n\geq1\mid\BE^n(-,M)\neq0\}\bigr)$, respectively.}
\end{definition}

\begin{remark}
If $\Lambda$ has global dimension $m$, then the projective and injective dimensions of any object in $\dem{\Lambda}$ are at most $m+d-1$ by Proposition~\ref{prop:global_dim}. Conversely, {if $H^0(\Lambda)\neq0$ and} the projective dimensions of any object in $\dem{\Lambda}$ are at most $n$, then $\gldim \Lambda \leq n-d+1$.
\end{remark}

\begin{proposition}[{\cite{tom2025,min2021}}]
Let $N\in\dem{\Lambda}$. {The projective dimension $\pd N$ coincides with the infimum of the set of non-negative integers $n$ such that $N\in\MK^{[-n,0]}(\MP_\Lambda)$, with the convention $\inf\varnothing=\infty$.}
\end{proposition}

\begin{proposition}\label{prop:inj_dim_d+1}
The following hold:
\begin{enumerate}
	\item Let $P^\bullet$ be an object in $\MK^{[-d,0]}(\MP_\Lambda)$. If $\tau^{>- d} P^\bullet\in\dem{\Lambda}$ has projective dimension {at most} $d+1$, then $(\tau^{\leq -d} P^\bullet)[-d]\in\dem{\Lambda}$ is projective.
	\item Let $I^\bullet$ be an object in $\MK^{[0,d]}(\MI_\Lambda)$. If $\tau^{\leq 0} I^\bullet\in\dem{\Lambda}$ has injective dimension {at most} $d+1$, then $(\tau^{> 0} I^\bullet)[d]\in\dem{\Lambda}$ is injective.
\end{enumerate}
\end{proposition}
\begin{proof}
We only show the first assertion. Consider the triangle below in $\pvd\Lambda$:
\begin{equation}\label{eq:trunc}
\tau^{\leq -d} P^\bullet \to P^\bullet \to \tau^{> -d} P^\bullet \dashrightarrow.
\end{equation}
By applying the functor $\pvd\Lambda(-,M[d+1])$ to triangle \eqref{eq:trunc} for any $M\in\dem{\Lambda}$, we have the following exact sequence in $\mod k$:
$$\pvd\Lambda(P^\bullet,M[d+1])\to \pvd\Lambda(\tau^{\leq -d}P^\bullet,M[d+1])\to \pvd\Lambda(\tau^{> -d} P^\bullet,M[d+2]).$$
By Lemma~\ref{lem:proj_pres} and the assumption, the leftmost and rightmost terms are zero. Thus we have
$$\BE\bigl((\tau^{\leq -d}P^\bullet)[-d], M\bigr)\cong\pvd\Lambda\bigl((\tau^{\leq -d} P^\bullet)[-d],M[1]\bigr)=0$$
for any $M\in\dem{\Lambda}$. This implies that $(\tau^{\leq -d} P^\bullet)[-d]$ is projective.
\end{proof}

For the rest of this subsection, we explain the notion of dominant dimensions of $d$-truncated DG-algebras and Auslander $d$-truncated DG-algebras. We fix a $d$-truncated proper connective DG-algebra $\Gamma$.

\begin{definition}\label{dfn:dominant_dim}
Let $n$ be a positive integer. We say that $\Gamma$ has \emph{dominant dimension at least $n$} if there exists a sequence of conflations in $\dem{\Gamma}$:
$$\Gamma\to I_0\to M_1\dashrightarrow;$$
$$M_1\to I_1\to M_2\dashrightarrow;$$
$$\vdots$$
$$M_{n-1}\to I_{n-1}\to M_n\dashrightarrow,$$
where $I_i$ are projective-injective objects in $\dem{\Gamma}$ for any $0\leq i\leq n-1$. The \emph{dominant dimension} $\domdim \Gamma$ of $\Gamma$ is defined as the supremum of the set of integers $n$ such that $\Gamma$ has dominant dimension at least $n$. If there is no such integer, we set $\domdim \Gamma=0$. If this set is unbounded above, we set $\domdim \Gamma=\infty$.
\end{definition}

For convenience, we also introduce the notion of codominant dimension.

\begin{definition}\label{dfn:codom_dim}
Let $n$ be a positive integer. We say that $\Gamma$ has \emph{codominant dimension at least $n$} when there exists a sequence of conflations in $\dem{\Gamma}$:
$$M_{-1}\to P_{0}\to D_d\Gamma\dashrightarrow;$$
$$M_{-2}\to P_{-1}\to M_{-1}\dashrightarrow;$$
$$\vdots$$
$$M_{-n}\to P_{-n+1}\to M_{-n+1}\dashrightarrow,$$
where ${P_{-i}}$ are projective-injective objects in $\dem{\Gamma}$ for any $0\leq i\leq n-1$. The \emph{codominant dimension} $\codomdim \Gamma$ of $\Gamma$ is defined as the supremum of the set of integers $n$ such that $\Gamma$ has codominant dimension at least $n$. If there is no such integer, we set $\codomdim \Gamma=0$. If this set is unbounded above, we set $\codomdim \Gamma=\infty$.
\end{definition}

\begin{remark}
This definition agrees with the dominant dimension of extriangulated categories in \cite[Definition~3.5]{gor2023}; for $d=1$ it recovers the usual dominant dimension of finite-dimensional algebras introduced in \cite{MR104718,MR161888}.
\end{remark}

\begin{remark}\label{rmk:domdim}
The definition of dominant dimension depends on the fixed integer $d$. That is, for a proper $d$-truncated DG-algebra $\Gamma$, the dominant dimension $\domdim \Gamma$ may change if we regard $\Gamma$ as a proper $d'$-truncated DG-algebra for some $d'\geq d$. The global dimension of $\Gamma$ is, by definition, independent of $d$.
\end{remark}

\begin{lemma}\label{lem:domdim}
$\Gamma$ has dominant dimension at least $d+1$ if and only if there exists an object $I^\bullet\in \MK^{[0,d]}(\MI_\Gamma\cap\MP_\Gamma)$ such that $\tau^{\leq 0} I^\bullet\cong \Gamma$ in $\dem{\Gamma}$. Dually, $\Gamma$ has codominant dimension at least $d+1$ if and only if there exists an object $P^\bullet\in \MK^{[-d,0]}(\MI_\Gamma\cap\MP_\Gamma)$ such that $\tau^{> -d} P^\bullet\cong {D_d\Gamma}$ in $\dem{\Gamma}$.
\end{lemma}
\begin{proof}
This follows directly from Proposition~\ref{prop:inductive_cocone}.
\end{proof}

\begin{definition}\label{dfn:auslander}
The $d$-truncated DG-algebra $\Gamma$ is called \emph{Auslander} if the following equivalent conditions are satisfied:
\begin{itemize}[labelwidth=1.5em, align=right]
\item[(i)] $\dem{\Gamma}$ is a $d$-Auslander extriangulated category.
\item[(ii)] $\gldim (\dem{\Gamma}) \leq d+1\leq \domdim (\dem{\Gamma}).$
\item[(ii')] $\gldim (\dem{\Gamma}) \leq d+1\leq \codomdim (\dem{\Gamma}).$ 
\item[(iii)] $\gldim \Gamma \leq 2$ and there exists an object $I^\bullet\in \MK^{[0,d]}(\MI_\Gamma\cap\MP_\Gamma)$ such that $\tau^{\leq 0} I^\bullet\cong \Gamma$ in $\dem{\Gamma}$.
\item[(iii')] $\gldim \Gamma \leq 2$ and there exists an object $P^\bullet\in \MK^{[-d,0]}(\MI_\Gamma\cap\MP_\Gamma)$ such that $\tau^{> -d} P^\bullet\cong D_d\Gamma$ in $\dem{\Gamma}$. 
\end{itemize}
\end{definition}

The equivalence between (i) and (ii) is precisely the definition of $d$-Auslander extriangulated categories. The equivalence of (ii) and (iii) follows from Proposition~\ref{prop:global_dim} and Lemma~\ref{lem:domdim}. The equivalence of (ii) and (ii') follows from the self-duality of $d$-Auslander extriangulated categories (see \cite[{Proposition~3.4}]{che2023c}). 

\begin{remark} Note the following for Auslander $d$-truncated DG-algebras:
\begin{enumerate}
	\item If $d=1$, this definition coincides with {the definition of} Auslander algebras. 
	\item The definition of Auslander $d$-truncated DG-algebras depends on the fixed integer $d$. Namely, for a proper $d$-truncated DG-algebra $\Gamma$, it may be an Auslander $d$-truncated DG-algebra but not an Auslander $d'$-truncated DG-algebra for some $d'\geq d$. See also Remark~\ref{rmk:domdim}.
\end{enumerate}
\end{remark}

\subsection{\texorpdfstring{Auslander correspondence for $d$-truncated DG-algebras}{Auslander d-truncated DG-algebras}}

To state the main result, we introduce finite representation type for $d$-truncated DG-algebras.

\begin{definition}
Let $\Lambda$ be a proper $d$-truncated DG-algebra. We say that $\Lambda$ is \emph{representation-finite} if there exists an object $M\in\dem{\Lambda}$ such that $\dem{\Lambda}=\add M$. Such an object $M$ is called an \emph{additive generator} of $\dem{\Lambda}$. If $M$ is basic in $\dem{\Lambda}$ (i.e., it has no repeated indecomposable summands), then we call it a \emph{basic additive generator}.
\end{definition}

\begin{remark}
Since $\dem{\Lambda}$ is Krull--Schmidt, if $\Lambda$ is representation-finite, then there exists a basic additive generator of $\dem{\Lambda}$ up to isomorphism.
\end{remark}

We also introduce Morita equivalence for $d$-truncated DG-algebras.

\begin{definition}\label{dfn:morita}
Let $\Lambda$ and $\Lambda'$ be $d$-truncated proper connective DG-algebras. 
We say that $\Lambda$ and $\Lambda'$ are \emph{Morita equivalent} if $\demdg{\Lambda}$ and $\demdg{\Lambda'}$ are quasi-equivalent as DG-categories (see Definition~\ref{dfn:can_enhance} for $\demdg{\Lambda}$ and $\demdg{\Lambda'}$).
\end{definition}

\begin{remark}\label{rmk:morita}
Let $\Lambda$ and $\Lambda'$ be $d$-truncated proper connective DG-algebras. It is straightforward to see that $\Lambda$ and $\Lambda'$ are Morita equivalent if and only if the DG-categories $\mathscr{P}_\Lambda$ and $\mathscr{P}_{\Lambda'}$ are quasi-equivalent, where $\mathscr{P}_\Lambda$ (resp. $\mathscr{P}_{\Lambda'}$) is the full DG-subcategory of $\demdg{\Lambda}$ (resp. $\demdg{\Lambda'}$) consisting of the projective objects in $\dem{\Lambda}$ (resp. $\dem{\Lambda'}$).
\end{remark}

\begin{remark} Let $\Lambda$ and $\Lambda'$ be $d$-truncated proper connective DG-algebras. If $\Lambda$ and $\Lambda'$ are Morita equivalent, then $\Lambda$ is of finite representation type (resp. Auslander) if and only if so is $\Lambda'$. 
\end{remark}

The next theorem is the main result of this paper.

\begin{theorem}\label{thm:main_first}
There exists a bijection between the following two classes:
\begin{itemize}
\item[(i)] The Morita equivalence classes of $d$-truncated DG-algebras $\Lambda$ of finite representation type.
\item[(ii)] The Morita equivalence classes of Auslander $d$-truncated DG-algebras $\Gamma$.
\end{itemize}
The correspondence from (i) to (ii) is given by the endomorphism $d$-truncated DG-algebra $\Gamma:=\End_\Lambda(M)$ of an additive generator $M$ of $\dem{\Lambda}$, and the correspondence from (ii) to (i) is given by the $d$-truncated DG-algebra $\Lambda:=\End_\Gamma(P)$ where $P$ is a basic additive generator of the projective-injective objects in $\dem{\Gamma}$.
\end{theorem}

\subsubsection{The correspondence from (i) to (ii)}

Fix a proper $d$-truncated DG-algebra $\Lambda$ of finite representation type and an additive generator $M$ of $\dem{\Lambda}$. Set $\Gamma:=\End_\Lambda(M)$ for the endomorphism $d$-truncated DG-algebra of $M$ (see Definition~\ref{dfn:end_alg}). The aim here is to show that $\Gamma$ is Auslander.

\begin{remark}
We may view $M$ as a DG-bimodule over $\Gamma$ and $\Lambda$. This yields the functor
\[
{}_{\Lambda}(M,-)\colon\dem{\Lambda}\to\dem{\Gamma},
\]
which we denote by $F_M$.
\end{remark}

\begin{proposition}
Let $\Lambda$ be a proper $d$-truncated DG-algebra of finite representation type and $M$ be an additive generator of $\dem{\Lambda}$. Then the endomorphism $d$-truncated DG-algebra $\Gamma:=\End_\Lambda(M)$ is Auslander. 
\end{proposition}

\begin{proposition}\label{prop:fully_faithful}
The functor $F_M:={}_\Lambda(M,-)\colon \dem{\Lambda}\to \dem{\Gamma}$ is fully faithful. Moreover, its essential image corresponds to the full subcategory of projective objects $\MP_\Gamma\subset\dem{\Gamma}$.
\end{proposition}
\begin{proof}
This follows because $F_M$ sends $M$ to $\Gamma$ and induces the isomorphism
$$\dem{\Lambda}(M,M)\xrightarrow{\cong} \dem{\Gamma}\bigl(F_M(M),F_M(M)\bigr)=\dem{\Gamma}(\Gamma,\Gamma)=H^0(\Gamma)$$
because $M$ is an additive generator of $\dem{\Lambda}$ by assumption.
\end{proof}

\begin{lemma}\label{lem:non_retraction_chara}
Let $g\colon N'\to N$ be a morphism in $\dem{\Lambda}$ where $N$ is indecomposable. Then the following are equivalent:
\begin{enumerate}
\item The morphism $g\colon N'\to N$ is not a retraction.
\item The induced morphism 
$$g\circ - \colon \dem{\Lambda}(M,N') \to \dem{\Lambda}(M,N)$$
factors through the inclusion $\iota\colon \rad \bigl(\dem{\Lambda}(M,N)\bigr) \to \dem{\Lambda}(M,N)$ in $\mod H^0\Gamma$.
\item The induced morphism
$$g_* \colon {}_\Lambda(M,N') \to {}_\Lambda(M,N)$$
factors through the $1$-monomorphism $\iota\colon \rad {}_\Lambda(M,N)\to {}_\Lambda(M,N)$ in $\dem{\Gamma}$.
\end{enumerate}
\end{lemma}
\begin{proof}
First, we prove the equivalence of (1) and (2). The objects $\dem{\Lambda}(M,N')$ and $\dem{\Lambda}(M,N)$ in $\mod H^0\Gamma$ are projective because $M$ is an additive generator of $\dem{\Lambda}$. Moreover, $\dem{\Lambda}(M,N)$ is indecomposable since $N$ is. Hence the morphism $g\circ-\colon {H^0F_M(N')}\to {H^0F_M(N)}$ factors through the radical of ${H^0F_M(N)}$ if and only if $g\circ-$ is not surjective, equivalently, $g$ is not a retraction.

Next, we show the equivalence of (2) and (3). Since $\hom{\Lambda}{M,N'}$ is a projective object in $\dem{\Gamma}$ by Proposition~\ref{prop:fully_faithful}, we have the commutative diagram in $\mod k$:
$$
\begin{tikzcd}
\dem{\Gamma}(\hom{\Lambda}{M,N'}, \rad \hom{\Lambda}{M,N}) 
& \mod H^0\Gamma\bigl(H^0\bigl(\hom{\Lambda}{M,N'}\bigr), H^0\bigl(\rad \hom{\Lambda}{M,N}\bigr)\bigr)\\
\dem{\Gamma}(\hom{\Lambda}{M,N'}, \hom{\Lambda}{M,N})
& \mod H^0\Gamma\bigl(H^0\bigl(\hom{\Lambda}{M,N'}\bigr), H^0\bigl(\hom{\Lambda}{M,N}\bigr)\bigr)
\Ar{1-1}{1-2}{"\cong"}
\Ar{2-1}{2-2}{"\cong"}
\Ar{1-1}{2-1}{}
\Ar{1-2}{2-2}{}
\end{tikzcd}
$$
by Lemma~\ref{lem:homology_functor}. The right vertical map is identified with the following map
$${\mod H^0\Gamma\bigl(\dem{\Lambda}(M,N'),\rad\dem{\Lambda}(M,N)\bigr)\to\mod H^0\Gamma\bigl(\dem{\Lambda}(M,N'),\dem{\Lambda}(M,N)\bigr)}$$
Hence the morphism $g\circ -$ factors through the radical of $\dem{\Lambda}(M,N)$ in $\mod H^0\Gamma$ if and only if $g_*\colon {}_\Lambda(M,N')\to {}_\Lambda(M,N)$ factors through the radical of ${}_\Lambda(M,N)$ in ${\dem{\Gamma}}$. This proves the equivalence of (2) and (3).
\end{proof}

\begin{remark}
Since {$\iota\colon \rad\hom{\Lambda}{M,N}\to\hom{\Lambda}{M,N}$} is a $1$-monomorphism in $\dem{\Gamma}$, it is also a monomorphism in ${\dem{\Gamma}}$ by Corollary~\ref{cor:defl_chara}. Hence lifts of the morphism $g_*$ in Lemma~\ref{lem:non_retraction_chara} are unique in $\dem{\Gamma}$. We denote such a lift by the same symbol $g_*'$. Consequently, $\iota \circ g'_* = g_*$.
\end{remark}

\begin{lemma}\label{lem:sink_chara}
{Let $g\colon N'\to N$ be a morphism in $\dem{\Lambda}$ which is not a retraction, with $N$ indecomposable. Then $g$ is right almost split if and only if $g_*'\colon {}_\Lambda(M,N')\to\rad{}_{\Lambda}(M,N)$ is a $d$-epimorphism in $\dem{\Gamma}$.}
\end{lemma}
\begin{proof}
One checks that $g$ is {right almost split} if and only if the induced morphism 
$$g\circ -\colon \dem{\Lambda}(M,N')\to \rad\bigl(\dem{\Lambda}(M,N)\bigr)$$ 
is an epimorphism in $\mod H^0\Gamma$. 
This is equivalent to {$g_*'\colon {}_\Lambda(M,N')\to\rad{}_{\Lambda}(M,N)$ being a $d$-epimorphism} in $\dem{\Gamma}$ since {$H^0(g_*')=g\circ-$} and by Corollary~\ref{cor:defl_chara}.
\end{proof}

\begin{proposition}\label{prop:proj_resol_of_simple}
Let $N\in \dem{\Lambda}$ be an indecomposable non-projective object. Then there exists a conflation in $\dem{\Gamma}$ as follows:
$${}_{\Lambda}(M,L)\xrightarrow{f_*} {}_{\Lambda}(M,N')\xrightarrow{g_*'} \rad {}_{\Lambda}(M,N)\dashrightarrow,$$
where $L\xrightarrow{f} N'\xrightarrow{g} N\dashrightarrow$ is an Auslander-Reiten conflation in $\dem{\Lambda}$.
\end{proposition}
\begin{proof}
Consider the following diagram in ${\pvd\Gamma}$:
$$
\begin{tikzcd}
\hom{\Lambda}{M,N'} & \rad\bigl(\hom{\Lambda}{M,N}\bigr) &  \Cone(g_*') & {}\\
\hom{\Lambda}{M,N'} & \hom{\Lambda}{M,N} & \Cone(g_*) & {}\\
{} & \top\bigl(\hom{\Lambda}{M,N}\bigr) & \top\bigl(\hom{\Lambda}{M,N}\bigr) & {} \\
{} & {} & {} & {}
\Ar{1-1}{1-2}{"g_*'"}
\Ar{1-2}{1-3}{}
\Ar{1-3}{1-4}{dashrightarrow}
\Ar{2-1}{2-2}{"g_*"}
\Ar{2-2}{2-3}{}
\Ar{2-3}{2-4}{dashrightarrow}
\Ar{3-2}{3-3}{equal}
\Ar{1-1}{2-1}{equal}
\Ar{1-2}{2-2}{}
\Ar{1-3}{2-3}{"\iota"}
\Ar{2-2}{3-2}{}
\Ar{2-3}{3-3}{}
\Ar{3-2}{4-2}{dashrightarrow}
\Ar{3-3}{4-3}{dashrightarrow}
\end{tikzcd}
$$
Since the cohomology of $\top\bigl(\hom{\Lambda}{M,N}\bigr)$ is concentrated in degree zero, we have the following isomorphisms in $\mod k$:
$$H^i\bigl( \Cone (g_*') \bigr) \to H^i\bigl( \Cone (g_*) \bigr)$$
for any $i\leq -1$. This means the following morphism is an isomorphism:
$$\tau^{\leq 0}\bigl(\Cocone (g_*')\bigr) \to \tau^{\leq 0}\bigl(\Cocone (g_*)\bigr)$$
in ${\pvd\Gamma}$. On the other hand,
\begin{align*}
\hom{\Lambda}{M,L}&:={\tau^{\leq0}}\bigl(\rhom{\Lambda}{M,L}\bigr)\\
&\cong\tau^{\leq0}\bigl(\Cocone(g_*)\bigr).
\end{align*}
Moreover, by Lemma~\ref{lem:sink_chara}, we get $H^1\bigl(\Cocone(g_*')\bigr)=0$. Consequently, we obtain
\begin{equation*}
\Cocone(g_*')\cong\tau^{\leq 0}\bigl(\Cocone(g_*')\bigr) \cong {\tau^{\leq0}\bigl(\Cocone(g_*)\bigr)}\cong \hom{\Lambda}{M,L}.\qedhere
\end{equation*}
\end{proof}

\begin{proposition}
Let $P\in \dem{\Lambda}$ be an indecomposable projective object. Then there exists the following isomorphism:
$$\rad\bigl(\hom{\Lambda}{M,P}\bigr)\cong \hom{\Lambda}{M,\rad P}.$$
Moreover, we have the following conflation in $\dem{\Gamma}$:
$$\hom{\Lambda}{M,\rad P}\to {}_\Lambda(M,P)\to \top {}_\Lambda(M,P)\dashrightarrow.$$
\end{proposition}
\begin{proof}
Since $\rad P\to P$ is a $1$-monomorphism, the morphism ${}_\Lambda(M,\rad P)\to {}_\Lambda(M,P)$ is also a $1$-monomorphism by Lemma~\ref{lem:n-mono_epi_chara}. Thus, ${}_\Lambda(M,\rad P)\to \rad {}_\Lambda(M,P)$ is an isomorphism by Corollary~\ref{cor:balanced} and Lemma~\ref{lem:sink_chara}. 
\end{proof}

\begin{notation}
For an indecomposable object $N\in\dem{\Lambda}$, we denote by $S_N$ the $\top\bigl(\hom{\Lambda}{M,N}\bigr)$. Note that $S_N$ is a simple object in $\mod H^0(\Gamma)$.
\end{notation}

This leads to the following proposition. 

\begin{proposition}\label{prop:simple_pd}
We have the following:
\begin{enumerate}
\item The simple objects $S_N\in\dem{\Gamma}$ associated to indecomposable non-projective objects $N\in\dem{\Lambda}$ have projective dimension $2$ (see Definition~\ref{dfn:pd}).
\item The simple objects $S_P\in\dem{\Gamma}$ associated to indecomposable projective objects $P\in\dem{\Lambda}$ have projective dimension at most $1$.
\end{enumerate} 
In particular, we have $\gldim (\dem{\Gamma}) \leq d+1$.
\end{proposition}

\begin{lemma}\label{lem:simple_top}
Consider a simple object $S\in\mod H^0\Gamma$. Then there exists an indecomposable object $N\in\dem{\Lambda}$ such that $S\cong S_N$. 
\end{lemma}
\begin{proof}
Although well known, we include a proof for completeness. Take a projective cover $P\to S$ in $\mod H^0\Gamma$. Since $P$ is projective, there exists an object $N\in\dem{\Lambda}$ such that $P\cong {H^0\bigl({}_\Lambda(M,N)\bigr)}$ by Proposition~\ref{prop:fully_faithful}{ and Lemma~\ref{lem:homology_functor}}. Since $P$ is indecomposable, so is $N$. Hence there exists an isomorphism $S\cong S_N$.
\end{proof}

\begin{cor}\label{cor:gldim2}
$\Gamma$ has global dimension at most $2$. Or equivalently, $\gldim (\dem{\Gamma}) \leq d+1$.
\end{cor}
\begin{proof}
By Lemma~\ref{lem:simple_top}, any simple object in $\mod H^0\Gamma$ is isomorphic to $S_N$ for some indecomposable object $N\in\dem{\Lambda}$. Thus, the assertion follows from Proposition~\ref{prop:simple_pd}.
\end{proof}

\begin{proposition}
Let $I\in\dem{\Lambda}$ be an indecomposable injective object. Then the object $F_M(I):={}_\Lambda(M,I)$ is a projective-injective object in $\dem{\Gamma}$.
\end{proposition}
\begin{proof}
It suffices to treat the case $I=D_d\Lambda$. Since $M$ is an additive generator of $\dem{\Lambda}$ by assumption, we have a section $s\colon \Lambda\to M^{\oplus r}$ for some positive integer $r$ in $\dem{\Lambda}$. By applying the functor 
$$\hom{\Lambda}{-,M}\colon (\dem{\Lambda})\op\to \dem{\Gamma\op}$$
to the morphism $s$, we obtain a retraction 
$s^*\colon \Gamma^{\oplus r}\cong \hom{\Lambda}{M^{\oplus r},M}\to \hom{\Lambda}{\Lambda,M}\cong M$ in $\dem{\Gamma\op}$. Thus, $M$ is a projective object in $\dem{\Gamma\op}$. By $k$-duality, there is a natural isomorphism in $\dem{\Gamma}$:
$$\hom{\Lambda}{M,D_d\Lambda}\cong D_d M$$
{Since} $D_d M$ is an injective object in $\dem{\Gamma}$, we have that $\hom{\Lambda}{M,D_d\Lambda}$ is projective-injective in $\dem{\Gamma}$.
\end{proof}

\begin{lemma}
Let $N$ be an object in $\dem{\Lambda}$. If $F_M(N)$ is injective in $\dem{\Gamma}$, then $N$ is injective in $\dem{\Lambda}$.
\end{lemma}
\begin{proof}
Take a $d$-monomorphism $N\to I$ in $\dem{\Lambda}$ where $I$ is an injective object in $\dem{\Lambda}$. Applying the functor $F_M$ yields a $d$-monomorphism $F_M(N)\to F_M(I)$ in $\dem{\Gamma}$ by {Lemma}~\ref{lem:n-mono_epi_chara}. Since $F_M(N)$ is injective in $\dem{\Gamma}$, the morphism $F_M(N)\to F_M(I)$ is a split monomorphism. Thus, the morphism $N\to I$ is also a split monomorphism by Proposition~\ref{prop:fully_faithful}. Therefore the object $N$ is injective in $\dem{\Lambda}$.  
\end{proof}

Therefore, we obtain the following proposition.

\begin{proposition}
The functor $F_M\colon \dem{\Lambda}\to \dem{\Gamma}$ induces an equivalence between $\MI_\Lambda$ and $\MP_\Gamma\cap \MI_\Gamma$. Moreover, $\add \bigl(F_M(D_d\Lambda)\bigr)=\MP_\Gamma\cap \MI_\Gamma$ holds.
\end{proposition}

We next show that $\Gamma$ has dominant dimension at least $d+1$.
\begin{lemma}\label{lem:calc}
For any object $N$ in $\dem{\Lambda}$ and $X$ in $\pvd\Lambda$, we have the isomorphism:
$$\tauleq{0}\bigl(\rhom{\Lambda}{N,\tauleq{0} X}\bigr)\cong \tauleq{0}\bigl(\rhom{\Lambda}{N,X}\bigr)$$
where the morphism is induced by $\tauleq{0} X\to X$ in $\pvd\Lambda$.
\end{lemma}
\begin{proof}
This follows from the vanishing below for any $i\geq 0$:
\begin{equation*}
H^{-i}\bigl(\rhom{\Lambda}{N,\taugeq{1} X}\bigr)\cong \pvd\Lambda(N[i],\taugeq{1}X)=0\qedhere
\end{equation*}
\end{proof}

\begin{cor}
For any object $N$ in $\dem{\Lambda}$ and its injective copresentation $I\in \RK^{[0,d]}(\MI_\Lambda)$ (see Definition~\ref{dfn:proj_pres_obj}), $\rhom{\Lambda}{M,I}$ is a projective-injective copresentation of ${F_M(N)}$ in $\dem{\Gamma}$.
\end{cor}
\begin{proof}
We have to show that $\rhom{\Lambda}{M,I}$ lies in $\RK^{[0,d]}({\MI_\Gamma}\cap\MP_\Gamma)$ for any $I\in \RK^{[0,d]}(\MI_\Lambda)$. This follows from $\rhom{\Lambda}{M,D_d \Lambda}\cong \tauleq{0}\rhom{\Lambda}{M,D_d \Lambda}=\hom{\Lambda}{M,D_d \Lambda}$.

Let $N$ be any object in $\dem{\Lambda}$ and take its injective copresentation $I\in \RK^{[0,d]}(\MI_\Lambda)$. Then there exist the following isomorphisms:
$$F_M(N):= {}_\Lambda(M,N)= \tau^{\leq 0}\rhom{\Lambda}{M,N} \cong \tau^{\leq 0}\rhom{\Lambda}{M,\tauleq{0}I}\cong \tau^{\leq 0}\rhom{\Lambda}{M,I}$$
where the last isomorphism is obtained by Lemma~\ref{lem:calc}. 
\end{proof}

\begin{proposition}
$\Gamma$ has dominant dimension at least $d+1$.
\end{proposition}
\begin{proof}
Take an injective copresentation $I^\bullet\in K^{[0,d]}(\MI_\Lambda)$ of $M\in\dem{\Lambda}$. Then by the above corollary, $\rhom{\Lambda}{M,I^\bullet}$ is a projective-injective copresentation of ${F_M(M)}=\Gamma$ in $\dem{\Gamma}$. {Thus Lemma~\ref{lem:domdim} gives $\domdim\Gamma\geq d+1$.}
\end{proof}

\begin{proposition}\label{prop:1to2}
The following holds:
\begin{enumerate}
\item $\MP_\Gamma\cap \MI_\Gamma$ has an additive generator $F_M(D_d\Lambda)$ which satisfies $\End_{\Gamma}(F_M(D_d\Lambda))\cong \Lambda$.
\item $\Gamma$ is an Auslander $d$-truncated DG-algebra.
\end{enumerate}
\end{proposition}

We conclude this part by discussing the well-definedness of the correspondence from (i) to (ii). 

\begin{lemma}\label{lem:well-def}
The following holds:
\begin{enumerate}
\item Let $M'$ be another additive generator of $\dem{\Lambda}$ and put $\Gamma':=\End_{\Lambda}(M')$. Then the Auslander $d$-truncated DG-algebras $\Gamma$ and $\Gamma'$ are Morita equivalent.
\item Let $\Lambda'$ be a $d$-truncated proper connective DG-algebra which is Morita equivalent to $\Lambda$ and let $M'$ be an additive generator of $\dem{\Lambda'}$. Put $\Gamma':=\End_{\Lambda'}(M')$. Then the Auslander $d$-truncated DG-algebras $\Gamma$ and $\Gamma'$ are Morita equivalent.
\end{enumerate}
\end{lemma}
\begin{proof}
The second assertion immediately follows from the first one. We show the first assertion. Since $\add M=\add M'$ in $\dem{\Lambda}$, we have a quasi-equivalence between the DG-categories $\mathscr{P}_\Gamma$ and $\mathscr{P}_{\Gamma'}$ (see Remark~\ref{rmk:morita} for details). Thus $\Gamma$ and $\Gamma'$ are Morita equivalent by Remark~\ref{rmk:morita}.
\end{proof}

\subsubsection{The correspondence from (ii) to (i)}

In this part, we prove that for an Auslander $d$-truncated DG-algebra $\Gamma$, the endomorphism DG-algebra of the additive generator of $\MP_\Gamma\cap \MI_\Gamma$ is a $d$-truncated proper connective DG-algebra of finite type. We first prepare several lemmas for general $d$-truncated proper connective DG-algebras.

\begin{definition}
Let $\Gamma$ be a $d$-truncated proper connective DG-algebra and $P\in\MP_\Gamma$. We define the full subcategory $\MK^{[-d,0]}(\add P)$ of $\MK^{[-d,0]}(\MP_\Gamma)$ as follows:
$$\MK^{[-d,0]}(\add P):=\add P*\add P[1]*\cdots *\add P[d]\subset \MK^{[-d,0]}(\MP_\Gamma).$$
\end{definition}

\begin{lemma}\label{lem:equiv_pres}
Let $\Gamma$ be a $d$-truncated proper connective DG-algebra and $P\in\MP_\Gamma$. Put $\Lambda:=\End_\Gamma(P)$. Then the functor $\rhom{\Gamma}{P,-}\colon \pvd\Gamma\to \pvd\Lambda$ induces an equivalence: 
$$\MK^{[-d,0]}(\add P)\to \MK^{[-d,0]}(\MP_\Lambda).$$
\end{lemma}
\begin{proof}
This follows from the definition of $\MK^{[-d,0]}(\add P)$ and the fact that $\rhom{\Gamma}{P,-}$ sends $P$ to $\Lambda$.
\end{proof}

\begin{lemma}\label{lem_comm}
For any object $Q^\bullet$ in $\MK^{[-d,0]}(\add P)$, we have an isomorphism:
$$\hom{\Gamma}{P,\tau^{>-d}Q^\bullet}\cong \tau^{>-d} \bigl(\rhom{\Gamma}{P,Q^\bullet}\bigr)$$
\end{lemma}
\begin{proof}
Consider the following triangle in ${\pvd\Lambda}$:
$$\rhom{\Gamma}{P,\tau^{\leq -d} Q^\bullet}\to \rhom{\Gamma}{P,Q^\bullet}\to \rhom{\Gamma}{P,\tau^{>-d} Q^\bullet}\dashrightarrow.$$
{Since $P\in\add\Gamma$ and the cohomology of $\tau^{\leq -d}Q^\bullet$ is concentrated in degrees at most $-d$, we have} 
$$H^i\bigl(\rhom{\Gamma}{P,\tau^{\leq -d} Q^\bullet}\bigr)=0$$ for any ${i>-d}$. This implies the desired isomorphism.
\end{proof}

\begin{lemma}\label{lem:submorita_modd}
Let $P\in\MP_\Gamma$ and denote $\Lambda:=\End_\Gamma(P)$. Define $\MC_P\subset\dem{\Gamma}$ to be the full subcategory consisting of objects $X$ such that there exists $Q^\bullet\in\MK^{[-d,0]}(\add P)\subset\MK^{[-d,0]}(\MP_\Gamma)$ with $\tau^{>-d} Q^\bullet\cong X$. Then the functor $\hom{\Gamma}{P,-}\colon \dem{\Gamma}\to \dem{\Lambda}$ induces an equivalence $\MC_P\to \dem{\Lambda}$. 
\end{lemma}
\begin{proof}
First, we prove essential surjectivity. Take any object $N\in\dem{\Lambda}$. By Proposition~\ref{prop:proj_inj_present}, there is an object $Q^\bullet$ such that
\[
Q^\bullet\in\MK^{[-d,0]}(\MP_\Lambda),
\qquad\tau^{>-d}Q^\bullet\cong N.
\]
By Lemma~\ref{lem:equiv_pres}, there is {an object} $Q'^\bullet$ such that
\[
Q'^\bullet\in\MK^{[-d,0]}(\add P),
\qquad\rhom{\Gamma}{P,Q'^\bullet}\cong Q^\bullet.
\]
Thus, $\tau^{>-d}Q'^\bullet$ lies in $\MC_P$ and we have the following isomorphisms:
$$\hom{\Gamma}{P,\tau^{>-d} Q'^\bullet}\cong \tau^{>-d} \rhom{\Gamma}{P,Q'^\bullet}\cong \tau^{>-d} Q^\bullet\cong N$$
by Lemma~\ref{lem_comm}. Hence we have established essential surjectivity.

Next, we show that the functor $\hom{\Gamma}{P,-}\colon \MC_P\to \dem{\Lambda}$ is fully faithful. Take any object $L\in\MC_P$ and $Q^\bullet\in\MK^{[-d,0]}(\add P)$ with an isomorphism $\alpha\colon \tau^{>-d} Q^\bullet\to L$. Let $N$ be any object in $\dem{\Gamma}$. Then we have the following commutative diagram in $\pvd k$:
$$
\begin{tikzcd}
\hom{\Gamma}{L,N} & \hom{\Lambda}{\hom{\Gamma}{P,L}, \hom{\Gamma}{P,N}} \\
\hom{\Gamma}{\tau^{>-d} Q^\bullet, N} & \hom{\Lambda}{\hom{\Gamma}{P,\tau^{>-d} Q^\bullet}, \hom{\Gamma}{P,N}} \\
\tauleq{0}\bigl(\rhom{\Gamma}{Q^\bullet, N}\bigr) & \tauleq{0}\bigl(\rhom{\Lambda}{\rhom{\Gamma}{P,Q^\bullet}, \rhom{\Gamma}{P,N}}\bigr)
\Ar{1-1}{1-2}{}
\Ar{2-1}{2-2}{}
\Ar{1-1}{2-1}{"\alpha^*"}
\Ar{1-2}{2-2}{"\hom{\Gamma}{P,\alpha}^*" }
\Ar{3-1}{3-2}{}
\Ar{2-1}{3-1}{}
\Ar{2-2}{3-2}{}
\end{tikzcd}
$$
where the vertical maps are induced by $\alpha$ and the truncation morphism $Q^\bullet\to{\tau^{>-d}Q^\bullet}$. The horizontal maps are induced by the {functors} $\hom{\Gamma}{P,-}$ and $\rhom{\Gamma}{P,-}$. The vertical maps are isomorphisms because $\alpha$ is an isomorphism and by the dual statement of Lemma~\ref{lem:calc}. Thus, it suffices to show that the bottom horizontal map is an isomorphism. 

{Consider the comparison before applying $\tauleq{0}$. Both sides are exact contravariant functors of $Q^\bullet$. Since $\MK^{[-d,0]}(\add P)$ is obtained by extensions from $\add P[i]$ for $0\leq i\leq d$, it suffices to consider $Q^\bullet=P[i]$ in this range. In this case, we have natural isomorphisms:}
\begin{align*}
{\rhom{\Gamma}{P[i], N}} &{\cong \rhom{\Lambda}{\Lambda[i], \rhom{\Gamma}{P,N}}} \\
&{\cong \rhom{\Lambda}{\hom{\Gamma}{P,P}[i], \rhom{\Gamma}{P,N}}}
\end{align*}
{The composite of these isomorphisms is the comparison before truncation. Applying $\tauleq{0}$ shows that the bottom horizontal map is an isomorphism. Taking $H^0$ then proves that $\hom{\Gamma}{P,-}\colon \MC_P\to \dem{\Lambda}$ is fully faithful.}
\end{proof}

\begin{proposition}\label{prop:submorita_aus}
Let $\Gamma$ be an Auslander $d$-truncated DG-algebra and $P\in\MP_\Gamma\cap \MI_\Gamma$ be an additive generator of $\MP_\Gamma\cap \MI_\Gamma$. Put $\Lambda:=\End_\Gamma(P)$. Then the functor $\hom{\Gamma}{P,-}\colon \dem{\Gamma}\to \dem{\Lambda}$ induces an equivalence ${\MI_\Gamma}\to \dem{\Lambda}$. 
\end{proposition}
\begin{proof}
{By Lemma~\ref{lem:submorita_modd}, it is enough to show that $\MI_\Gamma=\MC_P$. The same lemma implies that $\MC_P$ is closed under finite direct sums and direct summands. Condition (iii') in Definition~\ref{dfn:auslander} gives a projective-injective presentation of $D_d\Gamma$, so $D_d\Gamma\in\MC_P$. Hence $\MI_\Gamma=\add D_d\Gamma\subset\MC_P$.} 

{Let $X\in\MC_P$ and choose $Q^\bullet\in\MK^{[-d,0]}(\add P)$ such that $\tau^{>-d}Q^\bullet\cong X$. Put $Y:=(\tau^{\leq-d}Q^\bullet)[-d]\cong\tauleq{0}(Q^\bullet[-d])\in\dem{\Gamma}$. Since $P$ is projective-injective, $Q^\bullet[-d]\in\MK^{[0,d]}(\MP_\Gamma\cap\MI_\Gamma)$ is an injective copresentation of $Y$. The bound $\gldim(\dem{\Gamma})\leq d+1$ implies $\id Y\leq d+1$. Proposition~\ref{prop:inj_dim_d+1}(2) therefore shows that $\tau^{>0}(Q^\bullet[-d])[d]\cong\tau^{>-d}Q^\bullet\cong X$ is injective. Thus $\MC_P\subset\MI_\Gamma$, as required.} 
\end{proof}

As a consequence of the proposition above, we obtain the following corollary.

\begin{cor}\label{cor:2to1}
Let $\Gamma$ be an Auslander $d$-truncated DG-algebra and $P\in \MP_\Gamma\cap \MI_\Gamma$ be an additive generator of $\MP_\Gamma\cap \MI_\Gamma$. Put $\Lambda:=\End_\Gamma(P)$. Then $\Lambda$ is a $d$-truncated proper connective DG-algebra of finite representation type with additive generator $\hom{\Gamma}{P,{D_d\Gamma}}\in \dem{\Lambda}$ of $\dem{\Lambda}$. Moreover, we have a quasi-isomorphism $\End_\Lambda\bigl(\hom{\Gamma}{P,{D_d\Gamma}}\bigr)\cong \Gamma$.
\end{cor}
\begin{proof}
{By Proposition~\ref{prop:submorita_aus}, the functor $\hom{\Gamma}{P,-}$ sends the additive generator $D_d\Gamma$ of $\MI_\Gamma$ to an additive generator of $\dem{\Lambda}$. The comparison of hom complexes in the proof of Lemma~\ref{lem:submorita_modd} and $k$-duality give quasi-isomorphisms of DG-algebras
\[
\End_\Lambda\bigl(\hom{\Gamma}{P,D_d\Gamma}\bigr)\simeq\End_\Gamma(D_d\Gamma)\simeq\Gamma.
\]}
\end{proof}

We recall Theorem~\ref{thm:main_first} and provide its proof.

\begin{theorem}[$=$ Theorem~\ref{thm:main_first}]\label{thm:main}
There exists a bijection between the following two classes:
\begin{itemize}
\item[(1)] The Morita equivalence classes of $d$-truncated proper connective DG-algebras $\Lambda$ of finite representation type.
\item[(2)] The Morita equivalence classes of Auslander $d$-truncated proper connective DG-algebras $\Gamma$.
\end{itemize}
The correspondence from (1) to (2) is given by $\Lambda\mapsto \End_\Lambda(M)$ where $M$ is an additive generator of $\dem{\Lambda}$. 
\end{theorem}
\begin{proof}
The correspondence from (1) to (2) is well-defined by Lemma~\ref{lem:well-def}. Its injectivity follows from Proposition~\ref{prop:1to2}, and its surjectivity follows from Corollary~\ref{cor:2to1}. This proves the theorem.
\end{proof}

As a corollary, we can characterize a connective DG-category $\mathscr{A}$ which is quasi-equivalent to the {canonical} DG-enhancement of a $d$-extended module category of a $d$-truncated proper connective DG-algebra of finite representation type as follows:

\begin{cor}\label{cor:main}
Let $\mathscr{A}$ be a connective DG-category with idempotent complete homotopy category $H^0(\mathscr{A})$. Then the following are equivalent:
\begin{itemize}
\item[(1)] $\mathscr{A}$ is quasi-equivalent to {the canonical DG-enhancement $\demdg{\Lambda}$ of $\dem{\Lambda}$} for some $d$-truncated proper connective DG-algebra $\Lambda$ of finite representation type.
\item[(2)] The homotopy category $H^0(\mathscr{A})$ has an additive generator $M$ such that its endomorphism DG-algebra $\End_{\mathscr{A}}(M):=\mathscr{A}(M,M)$ is an Auslander $d$-truncated proper connective DG-algebra. 
\end{itemize}
\end{cor}
\begin{proof}
The implication (1) $\Rightarrow$ (2) follows from Theorem~\ref{thm:main}. For (2) $\Rightarrow$ (1), since $H^0(\mathscr{A})$ is idempotent complete, there exists a quasi-equivalence $\mathscr{A}\to \mathscr{P}_\Gamma$ where $\Gamma:=\End_{\mathscr{A}}(M)$ by Remark~\ref{rmk:morita}. {Set $\Lambda:=\End_\Gamma(P)$, where $P$ is an additive generator of $\MP_\Gamma\cap\MI_\Gamma$. By Corollary~\ref{cor:2to1}, $\Lambda$ is of finite representation type. Let $\mathscr{I}_\Gamma$ be the full DG-subcategory of $\demdg{\Gamma}$ consisting of the injective objects in $\dem{\Gamma}$. The shifted derived Nakayama functor $(-\lox_\Gamma D\Gamma)[d-1]$ induces a quasi-equivalence between $\mathscr{P}_\Gamma$ and $\mathscr{I}_\Gamma$: it sends $\Gamma$ to $D_d\Gamma$, and the induced map $\Gamma\to\rend{\Gamma}{D_d\Gamma}$ is a quasi-isomorphism by $k$-duality. Moreover, the functor induced by $\hom{\Gamma}{P,-}$ is quasi-fully faithful by the comparison of hom complexes in the proof of Lemma~\ref{lem:submorita_modd}, and it is essentially surjective on $H^0$ by Proposition~\ref{prop:submorita_aus}. Hence $\mathscr{I}_\Gamma$ is quasi-equivalent to $\demdg{\Lambda}$, and so is $\mathscr{A}$.}
\end{proof}

{Finally, we describe the Auslander $2$-truncated DG-algebra associated with the path algebra $k(1\to2)$.} For simplicity, we assume that the base field $k$ is algebraically closed.

\begin{example}
Let $\Lambda$ be the path algebra given by the following quiver:
$$\begin{tikzcd}
1 \arrow[r] & 2
\end{tikzcd}$$
and we regard it as a $2$-truncated proper connective DG-algebra concentrated in degree $0$. Then the Auslander--Reiten quiver $\mathrm{AR}(\mathcal{H}^2_\Lambda)$ of $\mathcal{H}^2_\Lambda$ is given by the quiver {on the left below}:
{\footnotesize$$\begin{tikzcd}
&{
\begin{matrix}
1\\2
\end{matrix}
}
&&
{\color{red}
\begin{matrix}
2
\end{matrix}
}
&&
{\color{red}
\begin{matrix}
1
\end{matrix}
}\\
{
\begin{matrix}
2
\end{matrix}}
&&
{
\begin{matrix}
1
\end{matrix}}
&&
{\color{red}
\begin{matrix}
1\\2
\end{matrix}
}
\Ar{2-1}{1-2}{}
\Ar{1-2}{2-3}{}
\Ar{2-3}{1-4}{}
\Ar{1-4}{2-5}{}
\Ar{2-5}{1-6}{}
\tauAr{2-3}{2-1}{}
\tauAr{2-5}{2-3}{}
\tauAr{1-4}{1-2}{}
\tauAr{1-6}{1-4}{}
\end{tikzcd}
\quad \quad
\begin{tikzcd}[row sep=30pt]
& 2 && 4 && 6 \\
1 && 3 && 5
\Ar{2-1}{1-2}{"a"}
\Ar{1-2}{2-3}{"b"}
\Ar{2-3}{1-4}{"c"}
\Ar{1-4}{2-5}{"d"}
\Ar{2-5}{1-6}{"e"}
\Ar{2-1}{2-3}{"h_2",red}
\Ar{1-2}{1-4}{"h_1",red}
\Ar{2-3}{2-5}{"h_4",red}
\Ar{1-4}{1-6}{"h_3",red}
\end{tikzcd}
$$}
Thus, the Auslander $2$-truncated proper connective DG-algebra $\Gamma:=\End_\Lambda(M)$, where $M$ is a basic additive generator of $\mathcal{H}^2_\Lambda$, is quasi-isomorphic to the DG-quiver represented on the right-hand side above. {The arrows $a,b,c,d,e$ have degree $0$ and zero differential. The arrows $h_1,h_2,h_3,h_4$ have degree $-1$, with
\[
dh_2=b\circ a,\qquad dh_1=c\circ b,\qquad
dh_4=d\circ c,\qquad dh_3=e\circ d.
\]} The Auslander--Reiten quiver $\mathrm{AR}(\mathcal{H}^2_\Gamma)$ of $\mathcal{H}^2_\Gamma$ is given in Figure~\ref{fig:ARquiver_mod2Gamma} (Next page).
Here, the red numbers present the {dimension vectors} of cohomologies of degree $-1$. 

Using the Auslander--Reiten quiver, one can check that the extriangulated category $\mathcal{H}^2_\Gamma$ has global dimension $3$ and dominant dimension $3$. 
\end{example}

\begin{landscape}
\begin{midpage}

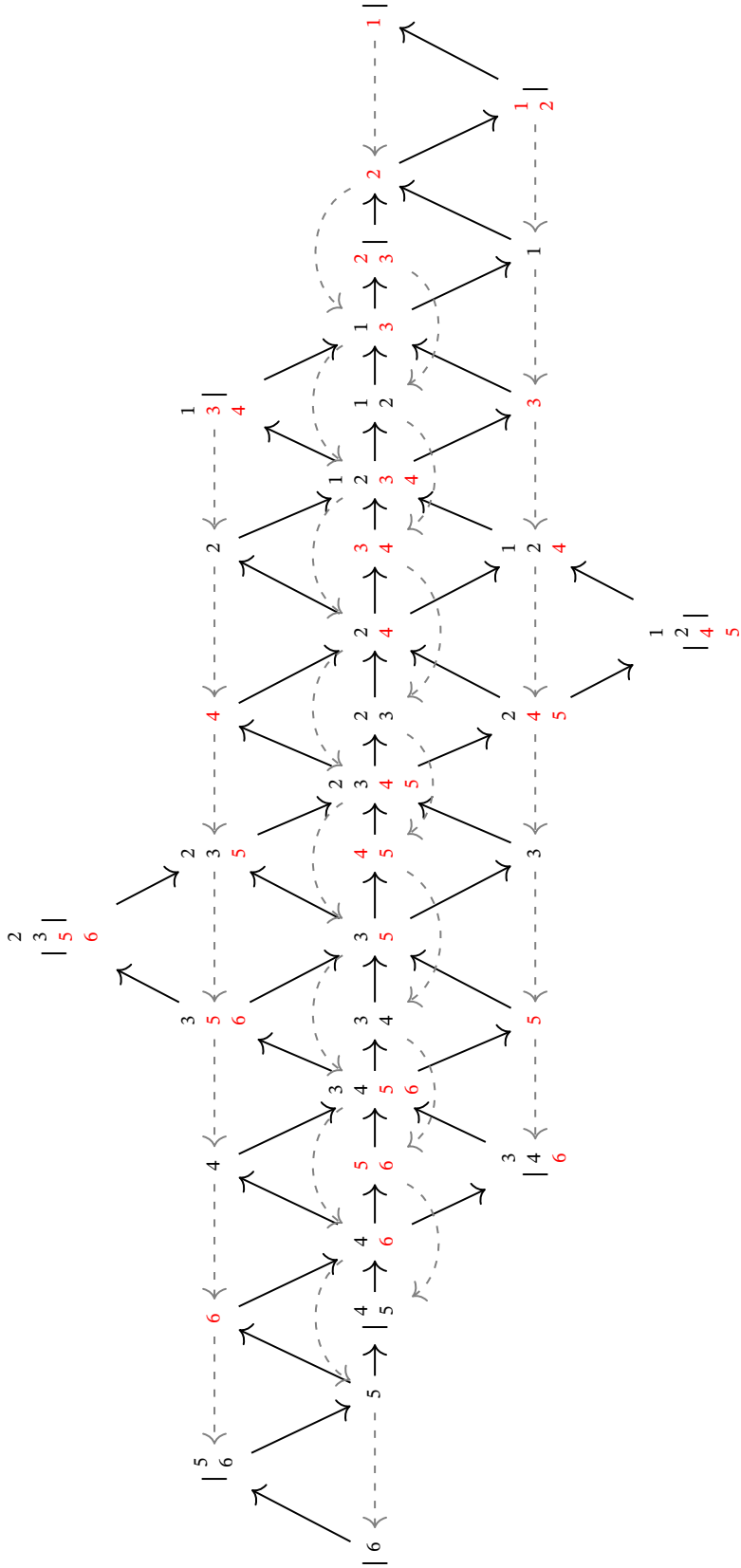
\begin{figure}[htbp]
\centering

{\scriptsize
$$
\begin{tikzcd}[column sep=small]
&&&&&&&&
{\mid
\begin{matrix}
2\\3\\\red 5\\\red 6
\end{matrix}
\mid}
\\
&
{\mid
\begin{matrix}
5\\6
\end{matrix}
}
&&
{
\begin{matrix}
\red 6
\end{matrix}
}
&&
{
\begin{matrix}
4
\end{matrix}
}
&&
{
\begin{matrix}
3\\\red 5\\\red 6
\end{matrix}
}
&&
{
\begin{matrix}
2\\3\\\red 5
\end{matrix}
}
&&
{
\begin{matrix}
\red 4
\end{matrix}
}
&&
{
\begin{matrix}
2
\end{matrix}
}
&&
{
\begin{matrix}
1\\\red 3\\\red 4
\end{matrix}\mid
}\\
{\mid
\begin{matrix}
6
\end{matrix}
}
&&
{
\begin{matrix}
5
\end{matrix}
}
&
{\mid
\begin{matrix}
4\\5
\end{matrix}
}
&
{
\begin{matrix}
4\\\red 6
\end{matrix}
}
&
{
\begin{matrix}
\red 5\\\red 6
\end{matrix}
}
&
{
\begin{matrix}
3\\4\\\red5\\\red6
\end{matrix}
}
&
{
\begin{matrix}
3\\4
\end{matrix}
}
&
{
\begin{matrix}
3\\\red5
\end{matrix}
}
&
{
\begin{matrix}
\red4\\\red5
\end{matrix}
}
&
{
\begin{matrix}
2\\3\\\red4\\\red5
\end{matrix}
}
&
{
\begin{matrix}
2\\3
\end{matrix}
}
&
{
\begin{matrix}
2\\\red4
\end{matrix}
}
&
{
\begin{matrix}
\red3\\\red4
\end{matrix}
}
&
{
\begin{matrix}
1\\2\\\red3\\\red4
\end{matrix}
}
&
{
\begin{matrix}
1\\2
\end{matrix}
}
&
{
\begin{matrix}
1\\\red3
\end{matrix}
}
&
{
\begin{matrix}
\red2\\\red3
\end{matrix}\mid
}
&
{
\begin{matrix}
\red2
\end{matrix}
}
&&
{
\begin{matrix}
\red1
\end{matrix}\mid
}
\\
&&&&&
{\mid
\begin{matrix}
3\\4\\\red6
\end{matrix}
}
&&
{
\begin{matrix}
\red5
\end{matrix}
}
&&
{
\begin{matrix}
3
\end{matrix}
}
&&
{
\begin{matrix}
2\\\red4\\\red5
\end{matrix}
}
&
&
{
\begin{matrix}
1\\2\\\red4
\end{matrix}
}
&&
{
\begin{matrix}
\red3
\end{matrix}
}
&&
{
\begin{matrix}
1
\end{matrix}
}
&&
{
\begin{matrix}
\red1\\\red2
\end{matrix}\mid
}\\
&&&&&&&&&&&&
{\mid
\begin{matrix}
1\\2\\\red4\\\red5
\end{matrix}\mid
}
\Ar{2-8}{1-9}{}
\Ar{1-9}{2-10}{}
\Ar{3-1}{2-2}{}
\Ar{2-2}{3-3}{}
\Ar{3-3}{2-4}{}
\Ar{2-4}{3-5}{}
\Ar{3-5}{2-6}{}
\Ar{2-6}{3-7}{}
\Ar{3-7}{2-8}{}
\Ar{2-8}{3-9}{}
\Ar{3-9}{2-10}{}
\Ar{2-10}{3-11}{}
\Ar{3-11}{2-12}{}
\Ar{2-12}{3-13}{}
\Ar{3-13}{2-14}{}
\Ar{2-14}{3-15}{}
\Ar{3-15}{2-16}{}
\Ar{2-16}{3-17}{}
\Ar{3-3}{3-4}{}
\Ar{3-4}{3-5}{}
\Ar{3-5}{3-6}{}
\Ar{3-6}{3-7}{}
\Ar{3-7}{3-8}{}
\Ar{3-8}{3-9}{}
\Ar{3-9}{3-10}{}
\Ar{3-10}{3-11}{}
\Ar{3-11}{3-12}{}
\Ar{3-12}{3-13}{}
\Ar{3-13}{3-14}{}
\Ar{3-14}{3-15}{}
\Ar{3-15}{3-16}{}
\Ar{3-16}{3-17}{}
\Ar{3-17}{3-18}{}
\Ar{3-18}{3-19}{}
\Ar{3-5}{4-6}{}
\Ar{4-6}{3-7}{}
\Ar{3-7}{4-8}{}
\Ar{4-8}{3-9}{}
\Ar{3-9}{4-10}{}
\Ar{4-10}{3-11}{}
\Ar{3-11}{4-12}{}
\Ar{4-12}{3-13}{}
\Ar{3-13}{4-14}{}
\Ar{4-14}{3-15}{}
\Ar{3-15}{4-16}{}
\Ar{4-16}{3-17}{}
\Ar{3-17}{4-18}{}
\Ar{4-18}{3-19}{}
\Ar{3-19}{4-20}{}
\Ar{4-20}{3-21}{}
\Ar{4-12}{5-13}{}
\Ar{5-13}{4-14}{}
\tauAr{2-4}{2-2}{}
\tauAr{2-6}{2-4}{}
\tauAr{2-8}{2-6}{}
\tauAr{2-10}{2-8}{}
\tauAr{2-12}{2-10}{}
\tauAr{2-14}{2-12}{}
\tauAr{2-16}{2-14}{}
\tauAr{3-21}{3-19}{}
\tauAr{3-19}{3-17}{bend right=60}
\tauAr{3-18}{3-16}{bend left=60}
\tauAr{3-17}{3-15}{bend right=60}
\tauAr{3-16}{3-14}{bend left=60}
\tauAr{3-15}{3-13}{bend right=60}
\tauAr{3-14}{3-12}{bend left=60}
\tauAr{3-13}{3-11}{bend right=60}
\tauAr{3-12}{3-10}{bend left=60}
\tauAr{3-11}{3-9}{bend right=60}
\tauAr{3-10}{3-8}{bend left=60}
\tauAr{3-9}{3-7}{bend right=60}
\tauAr{3-8}{3-6}{bend left=60}
\tauAr{3-7}{3-5}{bend right=60}
\tauAr{3-6}{3-4}{bend left=60}
\tauAr{3-5}{3-3}{bend right=60}
\tauAr{3-3}{3-1}{}
\tauAr{4-8}{4-6}{}
\tauAr{4-10}{4-8}{}
\tauAr{4-12}{4-10}{}
\tauAr{4-14}{4-12}{}
\tauAr{4-16}{4-14}{}
\tauAr{4-18}{4-16}{}
\tauAr{4-20}{4-18}{}
\end{tikzcd}
$$
}
\captionsetup{width=.8\textwidth}
\caption{The Auslander--Reiten quiver of $\mathcal{H}^2_{\Gamma}$ where $\Gamma$ is the Auslander $2$-truncated DG-algebra associated to $\Lambda$ in the previous example. The red numbers present the {dimension vectors} of cohomologies of degree $-1$. }
\label{fig:ARquiver_mod2Gamma}
\end{figure}

\end{midpage}

\end{landscape}

\bibliographystyle{mybstwithlabels}
\bibliography{myrefs}

\begin{thebibliography}{AMY19}
\expandafter\ifx\csname url\endcsname\relax
  \def\url#1{\texttt{#1}}\fi
\expandafter\ifx\csname doi\endcsname\relax
  \def\doi#1{\burlalt{doi:#1}{http://dx.doi.org/#1}}\fi
\expandafter\ifx\csname urlprefix\endcsname\relax\def\urlprefix{URL }\fi
\expandafter\ifx\csname href\endcsname\relax
  \def\href#1#2{#2}\fi
\expandafter\ifx\csname burlalt\endcsname\relax
  \def\burlalt#1#2{\href{#2}{#1}}\fi

\bibitem[AI12]{MR2927802}
T.~Aihara and O.~Iyama.
\newblock {\em Silting mutation in triangulated categories}.
\newblock J. Lond. Math. Soc. (2), {\bf 85}(3):633--668, 2012.

\bibitem[AIR14]{ada2014}
T.~Adachi, O.~Iyama, and I.~Reiten.
\newblock {\em $\tau$-tilting theory}.
\newblock Compositio Mathematica, {\bf 150}(3):415--452, 2014.

\bibitem[AMY19]{MR3898985}
T.~Adachi, Y.~Mizuno, and D.~Yang.
\newblock {\em Discreteness of silting objects and {$t$}-structures in triangulated categories}.
\newblock Proc. Lond. Math. Soc. (3), {\bf 118}(1):1--42, 2019.

\bibitem[Asa20]{MR4139031}
S.~Asai.
\newblock {\em Semibricks}.
\newblock Int. Math. Res. Not. IMRN, {\bf 2020}(16):4993--5054, 2020.

\bibitem[Aus71]{aus1971}
M.~Auslander.
\newblock {\em Representation dimension of {A}rtin algebras}. Queen Mary College, London.
\newblock Queen Mary College Mathematics Notes, 1971.

\bibitem[BR07]{bel2007}
A.~Beligiannis and I.~Reiten.
\newblock {\em Homological and homotopical aspects of torsion theories}.
\newblock Memoirs of the American Mathematical Society, {\bf 188}(883):viii+207, 2007.

\bibitem[Che23]{che2023}
X.~Chen.
\newblock {\em On exact dg categories}.
\newblock PhD thesis, Universit\'{e} Paris Cit\'{e}, 2023.
\newblock \url{https://arxiv.org/abs/2306.08231}.
\newblock arXiv:2306.08231.

\bibitem[Che25]{che2023c}
X.~Chen.
\newblock {\em Auslander--{I}yama correspondence for exact dg categories}.
\newblock Nagoya Mathematical Journal, {\bf 259}:499--523, 2025.

\bibitem[Che26a]{che2024}
X.~Chen.
\newblock {\em Exact dg categories {I}: Foundations}.
\newblock Advances in Mathematics, {\bf 489}:110809, 2026.

\bibitem[Che26b]{che2024b}
X.~Chen.
\newblock {\em Exact dg categories {II}: The embedding theorem}.
\newblock Journal of Pure and Applied Algebra, {\bf 230}(7):108298, 2026.

\bibitem[GNP23]{gor2023}
M.~Gorsky, H.~Nakaoka, and Y.~Palu.
\newblock {\em Hereditary extriangulated categories: Silting objects, mutation, negative extensions}.
\newblock arXiv preprint arXiv:2303.07134, 2023.

\bibitem[GNP26]{gor2021}
M.~Gorsky, H.~Nakaoka, and Y.~Palu.
\newblock {\em Positive and negative extensions in extriangulated categories}.
\newblock Mathematische Zeitschrift, {\bf 313}:89, 2026.

\bibitem[Gup24]{gup2024}
E.~Gupta.
\newblock {\em On $d$-term silting objects, torsion classes, and cotorsion classes}.
\newblock arXiv preprint arXiv:2407.10562, 2024.

\bibitem[GZ25]{gup2025}
E.~Gupta and Y.~Zhou.
\newblock {\em Semibricks and wide subcategories in extended module categories}.
\newblock arXiv preprint arXiv:2511.08157, 2025.

\bibitem[INP24]{iya2024}
O.~Iyama, H.~Nakaoka, and Y.~Palu.
\newblock {\em Auslander--{R}eiten theory in extriangulated categories}.
\newblock Transactions of the American Mathematical Society, Series B, {\bf 11}(08):248--305, 2024.

\bibitem[Iya07]{MR2298820}
O.~Iyama.
\newblock {\em Auslander correspondence}.
\newblock Adv. Math., {\bf 210}(1):51--82, 2007.

\bibitem[Jin20]{MR4133519}
H.~Jin.
\newblock {\em Cohen-{M}acaulay differential graded modules and negative {C}alabi-{Y}au configurations}.
\newblock Adv. Math., {\bf 374}:107338, 59, 2020.

\bibitem[Kel94]{kel1994}
B.~Keller.
\newblock {\em Deriving {DG} categories}.
\newblock Annales scientifiques de l'Ecole normale sup{\'e}rieure, {\bf 27}(1):63--102, 1994.

\bibitem[Kel06]{kel2006}
B.~Keller.
\newblock {\em On differential graded categories}.
\newblock In International Congress of Mathematicians. Vol. {II}, pages 151--190. European Mathematical Society, Z\"urich, 2006.

\bibitem[Min21]{min2021}
H.~Minamoto.
\newblock {\em Resolutions and homological dimensions of {DG}-modules}.
\newblock Israel Journal of Mathematics, {\bf 245}(1):409--454, 2021.

\bibitem[Moc25]{moc2025}
N.~Mochizuki.
\newblock {\em Higher-dimensional generalization of abelian categories via {DG}-categories}.
\newblock arXiv preprint arXiv:2501.06955, 2025.

\bibitem[MP26]{moc2025b}
N.~Mochizuki and M.~Plogmann.
\newblock {\em On the {A}uslander--{R}eiten theory for extended hearts of proper connective {DG} algebras}.
\newblock Bulletin of the London Mathematical Society, {\bf 58}(7):e70413, 2026.

\bibitem[Nak58]{MR104718}
T.~Nakayama.
\newblock {\em On algebras with complete homology}.
\newblock Abh. Math. Sem. Univ. Hamburg, {\bf 22}:300--307, 1958.

\bibitem[NP19]{nak2019}
H.~Nakaoka and Y.~Palu.
\newblock {\em Extriangulated categories, {H}ovey twin cotorsion pairs and model structures}.
\newblock Cah. Topol. G{\'e}om. Diff{\'e}r. Cat{\'e}g, {\bf 60}(2):117--193, 2019.
\newblock \url{https://www.lamfa.u-picardie.fr/palu/Nakaoka_Palu-LX-2.pdf}.

\bibitem[Run26]{run2026}
E.~S. Rundsveen.
\newblock {\em A note on the $m$-extended module categories of {N}akayama algebras}.
\newblock arXiv preprint arXiv:\allowbreak 2601.14843, 2026.

\bibitem[Ste23]{ste2023}
G.~Stefanich.
\newblock {\em Derived $\infty $-categories as exact completions}.
\newblock arXiv preprint arXiv:2310.12925, 2023.

\bibitem[Tac64]{MR161888}
H.~Tachikawa.
\newblock {\em On dominant dimensions of {${\rm QF}$}-3 algebras}.
\newblock Trans. Amer. Math. Soc., {\bf 112}:249--266, 1964.

\bibitem[Tom25]{tom2025}
R.~Tomonaga.
\newblock {\em On silting mutations preserving global dimension}.
\newblock arXiv preprint arXiv:2510.26206, 2025.

\bibitem[Zho24]{zho2024}
Y.~Zhou.
\newblock {\em Tilting theory for extended module categories}.
\newblock arXiv preprint arXiv:2411.15473, 2024.

\end{thebibliography}

\end{document}